\documentclass{llncs}

\usepackage{tikz-cd} \usepackage{amssymb} \usepackage{amsmath}
\usepackage{mathtools}
\usepackage{stmaryrd}
\usepackage{mathpartir}
\usepackage{bbm}
\usepackage{cleveref}
\usepackage[utf8]{inputenc}

\begin{document}

\newif\ifextended
\extendedtrue
\newcommand{\theappendix}{\ifextended the appendix \else the extended version \cite{extended-version}\fi}

\newtheorem{construction}{Construction}
\newcommand{\vett}{VETT}
\newcommand{\ohol}{\vett}
\newcommand{\fVDC}{\textrm{fVDCs}}
\newcommand{\VMJ}{\textrm{VM}_J}
\newcommand{\sem}[1]{\llbracket{#1}\rrbracket}
\newcommand{\pto}{\nrightarrow}
\newcommand{\pfrom}{\nleftarrow}
\newcommand{\vcat}{\mathcal}
\newcommand{\cat}{\mathbbm}
\newcommand{\isaSet}{\,\,\textrm{Set}}
\newcommand{\isaTy}{\,\,\textrm{Type}}
\newcommand{\isaCat}{\,\,\textrm{Cat}}
\newcommand{\isSmall}{\,\,\textrm{Small}}
\newcommand{\Set}{\textrm{Set}}
\newcommand{\MonCat}{\textrm{MonCat}}
\newcommand{\SymMonCat}{\textrm{SymMonCat}}
\newcommand{\Syn}{\textrm{Syn}}
\newcommand{\vtkmnd}{\mathbb{K}\text{Mod} (\vcat{V},T)}
\newcommand{\rmod}{\text{RMod}}
\newcommand{\lmod}{\text{LMod}}

\newcommand{\smallCats}{\text{SmallCat}}
\newcommand{\Cats}{\text{Cat}}
\newcommand{\varr}[2]{\text{Fun}\,{#1}\,{#2}}
\newcommand{\harr}[2]{\text{Prof}\,{#1}\,{#2}}
\newcommand{\harrapp}[3]{#1(#2;#3)}

\newcommand{\defaultObCtx}{\Gamma\pipe\alpha:\cat C}

\newcommand{\jnctx}{\curlyveedownarrow}

\newcommand{\id}{\textrm{id}}
\newcommand{\for}{\textrm{for}\,}
\newcommand{\when}{\textrm{when}\,}
\newcommand{\lett}{\textrm{let}\,}
\newcommand{\Sort}{\textrm{Sort}}
\newcommand{\isadtctx}{\,\,\textrm{type context}}
\newcommand{\isavectx}{\,\,\textrm{trans. context}}
\newcommand{\boundary}{\,\,\textrm{set context}}
\newcommand{\prof}{\,\,\textrm{span}}
\newcommand{\subst}{\,\,\textrm{subst}}
\newcommand{\sigctx}{\,\,\textrm{sig-ctx}}
\newcommand{\sig}{\,\,\textrm{sig}}
\newcommand{\pipe}{\mathrel{|}}

\newcommand{\punitinXfromYtoZ}[3]{#2 \mathop{\to_{#1}} #3}
\newcommand{\punitrefl}[1]{\textrm{id}_{#1}}
\newcommand{\punitelimtoYwithkontZ}[2]{\textrm{ind}_{\to}(#1,#2)}
\newcommand{\punitelimkontZatABC}[4]{\textrm{ind}_{\to}(#1,#2, #3, #4)}

\newcommand{\odotexists}[1]{\mathop{\overset{\exists #1}\odot}}
\newcommand{\tensorexistsXwithYandZ}[3]{#2 \odotexists{#1} #3}
\newcommand{\tensorintroofXandY}[2]{\textrm{pair}_\odot(#1,#2)}
\newcommand{\tensorintroatXwithYandZ}[3]{(#2,#1,#3)}
\newcommand{\tensorelimWkontZ}[2]{\textrm{ind}_{\odot}(#2;#1)}

\newcommand{\tlwith}[1]{\mathop{\prescript{#1}{}\triangleleft}}
\newcommand{\tlforall}[1]{\tlwith{\forall #1}}

\newcommand{\trwith}[1]{\mathop{\triangleright^{#1}}}
\newcommand{\trforall}[1]{\trwith {\forall #1}}

\newcommand{\homrallXYtoZ}[3]{#2 \trforall {#1} #3}

\newcommand{\homlallXYtoZ}[3]{#3 \tlforall {#1} #2}

\newcommand{\homrlambdaXatYdotZ}[3]{\lambda^\triangleright (#1, #2). #3}
\newcommand{\homllambdaXatYdotZ}[3]{\lambda^\triangleleft (#2, #1). #3}

\newcommand{\homrappXtoYatZ}[3]{#1 \trwith{#3} #2}
\newcommand{\homlappXtoYatZ}[3]{#2 \tlwith{#3} #1}

\newcommand{\homunary}[4]{\pendallXdotY {#1} {\homrallXYtoZ {#3} {#2} {#4}}}
\newcommand{\hombinary}[6]{\pendallXdotY {#1} {\homrallXYtoZ {#3} {#2} \homrallXYtoZ {#5} {#4} {#6}}}
\newcommand{\homtrinary}[8]{\pendallXdotY {#1} {\homrallXYtoZ {#3} {#2} \homrallXYtoZ {#5} {#4} \homrallXYtoZ {#7} {#6} {#8}}}

\newcommand{\homlunary}[4]{\pendallXdotY {#1} {\homrallXYtoZ {#3} {#2} {#4}}}
\newcommand{\homlbinary}[6]{\pendallXdotY {#1} {\homrallXYtoZ {#3} {#2} \homrallXYtoZ {#5} {#4} {#6}}}
\newcommand{\homltrinary}[8]{\pendallXdotY {#1} {\homrallXYtoZ {#3} {#2} \homrallXYtoZ {#5} {#4} \homrallXYtoZ {#7} {#6} {#8}}}

\newcommand{\lambdaunary}[4]{\pendlambdaXdotY {#1} {\homrlambdaXatYdotZ {#2} {#3} {#4}}}
\newcommand{\lambdabinary}[6]{\lambdaunary{#1}{#2}{#3}{\homrlambdaXatYdotZ{#4}{#5}{#6}}}
\newcommand{\lambdatrinary}[8]{\lambdabinary{#1}{#2}{#3}{#4}{#5}{\homrlambdaXatYdotZ{#6}{#7}{#8}}}

\newcommand{\appunary}[4]{\homrappXtoYatZ {\pendappXtoY {#1} {#2}} {#3} {#4}}
\newcommand{\appbinary}[6]{\homrappXtoYatZ {\appunary{#1}{#2}{#3}{#4}}{#5}{#6}}
\newcommand{\apptrinary}[8]{\homrappXtoYatZ {\appbinary{#1}{#2}{#3}{#4}{#5}{#6}}{#7}{#8}}

\newcommand{\pendallXdotY}[2]{\forall #1. #2}
\newcommand{\pendlambdaXdotY}[2]{\lambda #1. #2}
\newcommand{\pendappXtoY}[2]{#1^{#2}}

\newcommand{\Cat}{\textrm{Cat}}
\newcommand{\Id}[3]{\textrm{Id} #1\,#2\,#3}
\newcommand{\vlambda}[1]{\lambda^{F}{#1}.}
\newcommand{\hlambda}[2]{\lambda^{P}{#1};{#2}.}
\newcommand{\ran}[1]{\textrm{Ran}_{#1}\,}
\newcommand{\weightedLimitDW}[2]{\textrm{lim}^{#2}{#1}\,}

\newcommand{\equalizeVofWbyXeqYatZ}[5]{\{ {#1} : {#2} \pipe {#3} = {#4} : {#5} \}}

\newcommand{\paramPresheaf}[1]{\mathcal P^{#1}}
\newcommand{\pmPresheaf}{\paramPresheaf\pm}
\newcommand{\posPresheaf}{\paramPresheaf+}
\newcommand{\negPresheaf}{\paramPresheaf-}
\newcommand{\negPresheafAppPtoX}[2]{#2\in #1}
\newcommand{\posPresheafAppPtoX}[2]{#1 \ni #2}
\newcommand{\negPresheafApp}{\negPresheafAppPtoX}
\newcommand{\posPresheafApp}{\posPresheafAppPtoX}

\newcommand{\graphProf}[3]{\sum_{{#1};{#2}} #3}

\newcommand{\algCtx}{\textrm{Ctx}}
\newcommand{\algSubst}{\textrm{Subst}}
\newcommand{\algTy}{\textrm{Type}}
\newcommand{\algTm}{\textrm{Term}}
\newcommand{\algCat}{\textrm{Cat}}
\newcommand{\algVarr}{\textrm{Functor}}
\newcommand{\algHCtx}{\textrm{ProfCtx}}
\newcommand{\algHSubst}{\textrm{ProfSubst}}
\newcommand{\algHarr}{\textrm{Prof}}
\newcommand{\algTrans}{\textrm{Trans}}
\newcommand{\algElts}{\textrm{Elts}}
\newcommand{\algEltsI}{\textrm{EltsI}}

\newcommand{\algLHom}{\textrm{LHom}}
\newcommand{\algLHomI}{\textrm{LHomI}}
\newcommand{\algLHomE}{\textrm{LHomE}}
\newcommand{\algRHom}{\textrm{RHom}}
\newcommand{\algRHomI}{\textrm{RHomI}}
\newcommand{\algRHomE}{\textrm{RHomE}}

\newcommand{\algUnit}{\textrm{Unit}}
\newcommand{\algUnitI}{\textrm{UnitI}}
\newcommand{\algUnitE}{\textrm{UnitE}}

\newcommand{\algTensor}{\textrm{Tensor}}
\newcommand{\algTensorI}{\textrm{TensorI}}
\newcommand{\algTensorE}{\textrm{TensorE}}

\newcommand{\algCatTy}{\textrm{Cat}}
\newcommand{\algCatQt}{\textrm{CatQt}}
\newcommand{\algCatUnqt}{\textrm{CatUnQt}}

\newcommand{\algTransTy}{\textrm{Trans}}
\newcommand{\algTransQt}{\textrm{TransQt}}
\newcommand{\algTransUnqt}{\textrm{TransUnQt}}

\newcommand{\algVarrQt}{\textrm{FunctQt}}
\newcommand{\algVarrUnqt}{\textrm{FunctUnQt}}

\newcommand{\quoth}[1]{\lceil{} {#1}\rceil{}}
\newcommand{\unquoth}[1]{\lfloor{} {#1}\rfloor{}}

\newcommand{\citet}[1]{\cite{#1}}
\newcommand{\citep}[1]{(\cite{#1})}

\title{A Formal Logic for Formal Category Theory \ifextended{(Extended Version)}}
\author{Max S. New\inst{1,2}\orcidID{0000-0001-8141-195X} and Daniel R. Licata\inst{2}}
\institute{University of Michigan \and Wesleyan University}

\maketitle

\begin{abstract}
  We present a domain-specific type theory for constructions and proofs
  in category theory. The type theory axiomatizes notions of category,
  functor, profunctor and a generalized form of natural
  transformations. The type theory imposes an ordered linear restriction
  on standard predicate logic, which guarantees that all functions
  between categories are functorial, all relations are profunctorial,
  and all transformations are natural by construction, with no separate
  proofs necessary. Important category-theoretic proofs such as the
  Yoneda lemma and Co-yoneda lemma become simple type-theoretic proofs
  about the relationship between unit, tensor and (ordered) function
  types, and can be seen to be ordered refinements of theorems in
  predicate logic.  The type theory is sound and complete for a
  categorical model in \emph{virtual equipments}, which model both
  internal and enriched category theory. While the proofs in our type
  theory look like standard set-based arguments, the syntactic
  discipline ensure that all proofs and constructions carry over to
  enriched and internal settings as well.
\end{abstract}

\section{Introduction}

Category theory is a branch of mathematics that studies
higher-dimensional typed algebraic structures.
Originally developed for applications to homological algebra, it was
quickly discovered that categorical structures were common in logic and
computer science.  Formal systems like logics, type theories and
programming languages typically have sound and complete models given by
notions of structured categories~\cite{Lawvere69,LambekScott,moggi}.
This Curry-Howard-Lambek correspondence~
applies to simply typed lambda calculus~\cite{LambekScott},
computational lambda calculus~\cite{moggi}, linear logic~\cite{girard}
dependent type theory~\cite{cartmell1986,seely_1984}, and many other type theories
designed based on category-theoretic semantics.
The syntax of a type theory should present an initial object in its
category of models, a category-theoretic reformulation of logical
soundness and completeness.

While this research program has been quite successful,
category-theoretic notions can be overwhelming for beginners.
In a traditional set-theoretic formulation, notions such as adjoint
functors and limits produce a proliferation of ``naturality'' and
``functoriality'' side-conditions that must be discharged.
For example, when constructing an adjoint pair of
functors between two categories, a na\"ive approach would define all of
the data of the action on objects, action on arrows, prove the
functoriality of such actions, as well as construct two families of
transformations, prove they are natural and then finally proving a pair
of equalities relating compositions of natural transformations.
Carrying out these proofs explicitly is quite tedious and many
newcomers are left with the impression that category theory is full of
long, but ultimately trivial constructions.
This complexity is compounded when moving from ordinary category
theory to enriched and internal category theory, where constructions
must be additionally proven continuous, monotone, etc, in addition to
natural or functorial.
However, these generalizations are often exactly what is needed for programming
language applications; for example, domain-, metric- and
step-index-enriched categories have been used to model recursive
programming languages and internal categories have been used to model
parametricity and gradual
typing~\cite{wand197913,topos-of-trees,reflexive-graphs,double-cats-gradual-typing}.

Fortunately, the tools of category theory itself can be employed to
simplify this complexity, specifically the tools of \emph{higher}
category theory.
As an analogy in differential calculus, when an adept analyst writes
down a function, they do not expand out the $\epsilon\mathord{-}\delta$
definition of continuity for a function and proceed from first
principles, but rather use certain \emph{syntactic principles} for
defining functions that are continuous by construction --- e.g. that
composition of continuous functions is continuous.
Similar principles apply to category theory itself: functors and
natural transformations are closed under composition and whiskering
operations, and experienced category theorists rely on these syntactic
principles to eliminate the tedium of explicit proofs.
In the case of category theory, these principles can be formalized
using algebraic structures such as 2-categories, bicategories, Yoneda
structures, (virtual) double categories, pro-arrow
equipments~\cite{bicategories,proarrow-equipments,STREET1978350,LEINSTER2002391,Cruttwell2010},
an approach known as \emph{formal category theory}.
In these structures, rather than defining notions of category, functor
and natural transformation from first principles, they are axiomatized
in a manner similar to how a category axiomatizes a notion of space
and homomorphism.
Proofs in formal category theory apply to enriched and internal
settings, which are instances of the formal axioms.
A downside is that these algebraic structures are quite complicated,
and practitioners typically employ either an algebraic combinator
syntax (formalized in \citet{CURIEN1986188}) or a
2-dimensional diagrammatic language that can be quite beautiful and
elegant, but is also somewhat removed from the traditional formulation
of category theory in terms of sets and functions.

In this work, we apply the techniques of categorical logic to define a
more familiar logical syntax for carrying out constructions and proofs
in formal category theory.  
We call the resulting theory \emph{virtual equipment type theory} (VETT)
as (hyperdoctrines of) \emph{virtual
  equipments}~\cite{LEINSTER2002391,Cruttwell2010}, a particular semantic model of
formal category theory, provide a sound and complete notion of model for
the theory.
VETT provides syntax for categories, functors, profunctors, and natural
transformations, which are defined using familiar term syntax and
$\beta\eta$ reasoning principles for $\lambda$-functions, bound
variables, tuples, etc.  
By adhering to a \emph{syntactic discipline}, the logic guarantees
that all functor terms are automatically functorial, and all natural
transformation terms are natural.
More specifically, the syntax for transformations is a kind of
\emph{indexed, ordered linear} $lambda$ calculus, where the indexing
ensures that transformations are correctly natural and the ordering and
linearity ensure that the proofs are valid in a large class of enriched
and internal categories, such as enrichment in a non-symmetric monoidal
category.
VETT provides an alternative to algebraic and
string-diagram syntaxes for working with virtual equipments,
similar to how the lambda calculus provides an alternative to
categorical combinators and string diagram calculi for cartesian
closed categories.


%

The syntax of \vett{} is an indexed, ordered linear, proof-relevant
variant of predicate logic over a unary type theory.
Just as a predicate logic has a notion of type, term, relation and
implication, \vett{} is based on four analogous category-theoretic
concepts: categories, functors, profunctors and natural transformations
of profunctors.
Categories are treated like types, and the unary functors we consider in
this paper are each represented by a term whose type is a category and
whose one free variable ranges over a category.
The analog of a relation is a \emph{profunctor} (defined below), which
is written like a set with free category variables.  Like the
restriction to unary functors, we restrict to profunctors with two free
variables.
The logic is proof-relevant in that the implications of relations are
generalized to natural transformations of profunctors, and we use a
$\lambda$-calculus notation to describe these ``proof terms''.
This analogy to predicate logic can be made formal: any construction in
\vett{} can be erased to a corresponding construction or proof in
predicate logic, as sets, functions, relations, and implication of
relations define a (somewhat degenerate) virtual equipment.

While the restricted syntax developed in this paper does not express
some important concepts such as functor categories or opposite
categories, the restriction is natural in that it corresponds exactly to
virtual equipments, a well-understood notion of model that can express a
great deal of fundamental results and constructions in category
theory~\cite{riehl_verity_2022,Shulman13}.
Moreover, we can work around these unary/binary restrictions to some
extent by viewing the type theory as a domain-specific language embedded
in a metalanguage.  For example, while we cannot talk about functor
categories, we can state a theorem that quantifies over functors using
the meta-language's ``external'' universal quantifier (which does not
have automatic functoriality/naturality properties).  To support this,
\vett{} includes a third layer, an extensional dependent type theory in
the style of Martin-L\"of type theory. All of our ordered predicate
logic judgments are also indexed by a context from this dependent
type theory, and the type theory includes universe types for categories,
functors, profunctors and natural transformations.  This allow us to
formalize theorems the object logic is too restrictive to encode, analogous to
2-level~\cite{voevodsky13hts,altenkirch+16strict,PALMGREN2019102715} or indexed type
theories~\cite{isaev21indexed,cervesatopfenning02llf,vakar15linear,krishnaswami+15linear}.


While we emphasize the applications to enriched and internal category
theory in this work, there is potential for more direct application to
programming language semantics.
Ordinary predicate logic is the foundation for proof-theoretic presentations of
logical relations, such as Abadi-Plotkin logic for parametricity and
LSLR and Iris for step-indexed logical relations proofs
\cite{abadi-plotkin,lslr,iris}.
We conjecture that \vett{} might similarly serve as the foundation for
a logic of \emph{ordered} structures, which abound in applications:
rewriting and approximation relations can both be modeled as orderings
and logical relations involving these structures are proven to respect
orderings; operational logical relations must be downward-closed and
approximation relations should satisfy transitivity.
Just as LSLR and Iris release the user from the syntactic burden of
explicit step-indexing, \vett{} may be used to release the user from
the syntactic burden of proving downward-closure or transitivity
side-conditions.
Additionally, \vett{} may serve as the basis of a future domain specific
proof assistant for category-theoretic proofs.  To pilot-test this, we
have formalized the syntax of \vett{} in Agda 2.6.2.2, using the rewrite
mechanism to make \vett's substitution and $\beta$-reduction rules
definitional
equalities.\footnote{\url{https://github.com/maxsnew/virtual-equipments/blob/master/agda/STC.agda}}
We have used this lightweight implementation to check a number of
examples.  

\textbf{Basics of Profunctors.}  While we assume the reader has some
background knowledge of category theory, we briefly define profunctors,
which are not included in many introductory texts.  Recall that a
category $\cat C$ has a collection of objects and morphisms with
identity and composition, and a functor $F : \cat C \to \cat D$ is a
function on objects and a function on morphisms that preserves identity
and composition.  A category can be thought of as a generalization of a
preordered set, which has a set of elements and a binary \emph{relation}
on its objects satisfying reflexivity and transitivity.  A category is
then a \emph{proof-relevant preorder}, where morphisms are the proofs of
ordering, and the reflexivity and transitivity proofs must satisfy identity and
unit equations. A functor is then a \emph{proof-relevant monotone
  function}.  Given categories $\mathcal C$ and $\mathcal D$, a
profunctor $R$ from $\mathcal C$ to $\mathcal D$, written $R : \cat C
\pto \cat D$ is a functor $R : {\cat C}^{o} \times {\cat D} \to \Set$\footnote{${\cat C}^o$ is the notation we use for the opposite category of $\cat C$}.
Because a profunctor outputs a $\Set$ rather than a proposition, it is
itself a \emph{proof-relevant relation}.  Thinking of categories as
proof-relevant preorders, functoriality says that the profunctor is
downward-closed in $\cat C$ and upward-closed in $\cat D$.  Given
profunctors $R, S : \cat C \pto \cat D$, a homomorphism from $R$ to $S$
is a natural transformation, which in the preordered setting is
simply an implication of relations.


Profunctors are very useful for formalizing category theory, but an
additional reason we make them a basic concept of \vett{} is that they
allow us to give a \emph{universal property} for the type of ``morphisms
in a category ${\cat C}$''.  This is analogous to how the $J$
elimination rule for the identity type in Martin-L\"of type theory gives
a universal property for morphisms in a groupoid (the special case of a
category where all morphisms are
invertible)~\cite{hofmann98groupoid,awodeywarren09identity,voevodsky06homotopy}. The reason
profunctors are useful for this purpose is that, for any category $\cat
C$, $\text{Hom}_{\cat C} : \cat C \pto \cat C$ is a profunctor. On
preorders this is just the preorder's ordering relation
itself. Moreover, the hom profunctor is the unit for a composition of
profunctors $R \odot S$ which is defined as a \emph{co-end}.
The composition of profunctors is a generalization of the composition of
relations, and just as the equality relation is the identity for the
composition of relations, the hom profunctor is the identity for this
composition.  The unit law for the hom profunctor can be seen as a
``morphism induction'' principle, analogous to the ``path induction''
used in homotopy type theory (though in this paper we consider only
ordinary 1-dimensional categories, not higher generalizations).

\textbf{Outline.}  In Section~\ref{sec:syntax} we introduce the syntax
of \vett{}.  In Section~\ref{sec:examples} we demonstrate how to use our
syntax for formal category theory.  In Section~\ref{sec:semantics}, we
develop some model theory for \vett{}, including a sound and complete
notion of categorical model and sound interpretation in virtual
equipments modeling ordinary, enriched and internal category theory.  In
Section~\ref{sec:discussion}, we discuss related type theories and
potential extensions.

\section{Syntax of VETT}
\label{sec:syntax}

In Figure~\ref{fig:hol-cats} we give a table summarizing the
relationship between the judgments and connectives of higher-order
predicate logic with our ordered variant. Due to the incorporation of
variance, some unordered concepts generalize to multiple different
ordered notions. For instance, covariant and contravariant presheaf
categories generalize the power set. Further, because we only have
binary relations rather than relations of arbitrary arity, we have
only restricted forms of universal and existential quantification
which come combined with implications and conjunctions.

\begin{figure}
  \begin{center}
    \begin{tabular}{ |c|c| } 
      \hline
      Higher-Order Logic & Virtual Equipment Type Theory \\
      \hline
      Set $X$ & Category $\mathbb{C}$\\
      $X \times Y$ & $\mathbb C \times \mathbb D$ \\ 
      1 & $1$ \\
      $\mathcal P X$ & $\mathcal P^+ X$ \text{ and } $\mathcal P^- X$ \\
      $\{ (x,y) \in X \times Y | R(x,y) \}$ & $\sum_{\alpha:C;\beta:D} R$\\
      Function $f(x:X) : Y$ & Functor/Object $\alpha:\cat C \vdash A : \cat D$\\
      Relation $R(x,y)$ & Profunctor/Set $\alpha:\cat C; \beta:\cat D \vdash R$\\
      $R \wedge Q$ & $R \times Q$\\
      $\top$ & $1$\\
      $\forall x. P \Rightarrow Q$ & $\homrallXYtoZ {\alpha:\cat C} P Q \text{ and } \homlallXYtoZ {\alpha:\cat C} P Q$ \\
      $\exists x. P \wedge Q$ & $\tensorexistsXwithYandZ {\alpha:\cat C} P Q$\\
      $x =_X y$ & $\punitinXfromYtoZ {\cat C} \alpha \beta$\\
      Proof $\forall \overrightarrow \alpha. R_1 \wedge \cdots \Rightarrow Q$ & Nat. Trans./Element $\alpha_1,x_1: R_1(\alpha_1,\alpha_2),\ldots \vdash t : Q$\\
      \hline
    \end{tabular}
  \end{center}
  \caption{Analogy between Higher-Order Logic and \vett{} Judgments and Connectives}
\label{fig:hol-cats}
\end{figure}

The syntactic forms of \vett{} are given in
Figure~\ref{fig:syntax}.
First, we have categories, which are analogous to sorts in a
first-order theory. We have $M$ a base sort, product and unit sorts,
as well as the graph of a profunctor and the negative and positive
presheaf categories.
Next, objects $a,b,c$ are the syntax for the \emph{functors} between
categories. We call them objects rather than functors, because in
type-theoretic style, a functor is viewed as a ``generalized object''
parameterized by an input variable $\alpha : \cat C$.
Next, sets $P,Q,R$ are the syntax for \emph{sets}. These sets denote
\emph{profunctors}, i.e., a categorification of relations. Similar to
functors, rather than writing profunctors as functions ${\cat
  C}^o\times {\cat D} \to \Set$, we write them as sets with a
contravariant variable $\alpha:\cat C$ and a covariant variable
$\beta:\cat D$. The sets we can define are the Hom-set, the tensor and
internal hom, as well as products of sets, profunctors applied to two
objects and elements of positive and negative presheaves. Finally we
have elements of sets, which correspond to natural transformations of
multiple inputs, where again we view natural transformations valued in
a profunctor as generalized elements of profunctors.

After these forms we have types and terms, which represent the
meta-language that we use to talk about categories/profunctors/natural
transformations.  In addition to standard dependent type theory with
$\Pi$ and $\Sigma$ and identity types, we have universes of
categories, functors, profunctors and natural transformations.

Finally we have several forms of context which are used in the theory.
The contexts $\Gamma$ of term variables with their types are as usual;
we write ``${\Gamma \isadtctx}$'' to indicate that a context is well-formed.
We name the remaining contexts after the judgements that they are used
by. The set contexts $\Xi$, which will be used
to type-check sets, contain object variables with their categories. The
two forms of set context are $\alpha : \cat C$, containing one variable
that can be used both contravariantly and covariantly, and $\alpha :
\cat C ; \beta : \cat D$, containing a contravariant variable $\alpha$
and covariant variable $\beta$.  Finally, the transformation contexts
$\Phi$ contain element variables with their sets, alternating with those
sets' object variables with their categories.  A typical $\Phi$ has the
shape
\[
\alpha_1 : \cat C_1, x_1:R_1(\alpha_1,\alpha_2), \alpha_2 : \cat C_2,
x_2:R_2(\alpha_2,\alpha_3), \ldots, R_n(\alpha_n,\alpha_{n+1}),
\alpha_{n+1} : \cat C_{n+1}
\]
and represents the composition of the ``relations'' $R_1, R_2, R_3,
\ldots, R_n$.  We write $d^-(\Phi)$ for the first category variable in
$\Phi$ (which we regard as the negative or contravariant position),
$d^+(\Phi)$ for the last category variable in $\Phi$ (which we regard
as the positive or covariant position) and use the notation $d^\pm\Xi$
with the same meaning. We write $\Phi_1 \jnctx \Phi_2$ for the append
of two transformation contexts, which is only well-formed when the
last variable in $\Phi_1$ is equal to the first variable in
$\Phi_2$. Formal inductive definitions are in the appendix, but
intuitively:
\[
\begin{array}{ccl}
d^-(\alpha_1 : \cat C_1, x_1:R_1(\alpha_1,\alpha_2), \ldots, x_n:R_n(\alpha_n,\alpha_n), \alpha_{n+1} : \cat C_{n+1}) & = & \alpha_1 : \cat C_1\\
d^+(\alpha_1 : \cat C_1, x_1:R_1(\alpha_1,\alpha_2), \ldots, x_n:R_n(\alpha_n,\alpha_n), \alpha_{n+1} : \cat C_{n+1}) & = & \alpha_{n+1}:\cat C_{n+1}\\
(\Phi_1,\beta:\cat D) \jnctx (\beta:\cat D, \Phi_2) & = & \Phi_1, \beta : \cat D, \Phi_2
\end{array}
\]

\begin{figure}[t]
  \begin{mathpar}
    \begin{array}{rccl}
      \textrm{Categories} &\cat C, \cat D, \cat E & ::= & \unquoth M \pipe \cat C \times \cat D \pipe \cat 1
      \pipe \graphProf \alpha \beta P \pipe \negPresheaf {\cat C} \pipe \posPresheaf {\cat C}
      \\
      \textrm{Objects} & a, b, c & ::= & \alpha \pipe M a \pipe (a,b)
      \pipe () \pipe \pi_i a \pipe (a_-, a_+, s) \pipe \pi_- a \pipe
      \pi_+ a \pipe \lambda \alpha:\cat C.R \\
      \textrm{Sets} & P,Q,R & ::= & \punitinXfromYtoZ {\cat C} a b \pipe \tensorexistsXwithYandZ {\beta} P Q \pipe \homrallXYtoZ \beta P Q \pipe \homlallXYtoZ \alpha R S 
      \pipe 1 \pipe P \times Q\\
      &&& \pipe \harrapp M a b \pipe \negPresheafAppPtoX a b \pipe \posPresheafAppPtoX a b\\
      \textrm{Elements} & s,t,u & ::= & 
    x \pipe 
    {\punitelimkontZatABC {\alpha. t}{b_1} s {b_2}} \pipe
    \punitrefl b \pipe 
    {\tensorelimWkontZ s {x,\beta,y.r}} \pipe 
    {\tensorintroatXwithYandZ b s t} \pipe
    {\homrappXtoYatZ s t a} \\
    &&&
    \pipe {\homrlambdaXatYdotZ {x}{\alpha} s}\pipe 
    {\homlappXtoYatZ t s a} \pipe 
    {\homllambdaXatYdotZ {x}{\alpha} s} \pipe
    \pi_i s \pipe 
    (s_1,s_2) \pipe
    () \pipe \pi_e a \pipe \pendappXtoY M b \\ 
      \textrm{Type} & A,B,C & ::= &
      \ldots \pipe \smallCats \pipe \Cat \pipe \varr{\cat C}{\cat D}\pipe \harr {\cat C} {\cat D}\pipe \pendallXdotY {\alpha:{\cat C}} R\\
      \textrm{Term} & L,M,N & ::= & 
      \ldots \pipe \lceil \cat C \rceil \pipe \lambda {\alpha:\cat C}. a \pipe \lambda ({\alpha:\cat C};{\beta: \cat D}). R \mid \lambda \alpha.t \\
      \textrm{Type Context} & \Gamma, \Delta & ::= & \cdot \pipe \Gamma, X : A\\
      \textrm{Set Context} & \Xi, Z & ::= & \alpha:\cat C \pipe \alpha:\cat C;\beta: \cat D\\
      \textrm{Trans. Context} & \Phi,\Psi & ::= & \alpha : \cat C \pipe \Phi , x : P,\beta:\cat D\\
    \end{array}
  \end{mathpar}
  \caption{\vett{} Syntactic Forms}
  \label{fig:syntax}
\end{figure}

Next, we overview our basic judgement forms.  We have
\begin{itemize}
\item Categories: ${\Gamma\vdash \cat C \isaCat}$, where ${\Gamma \isadtctx}$.
  
\item Objects/functors: ${\Gamma \pipe \alpha : \cat C \vdash a : \cat
  D}$, where $\Gamma \vdash \cat C\isaCat$ and $\Gamma \vdash \cat D
  \isaCat$.  Objects are typed with an input object variable $\alpha :
  \cat C$ and an output category $\cat D$; in the semantics, objects are
  modeled as functors $\cat C \to \cat D$.
  
\item Sets/profunctors: ${\Gamma \pipe \Xi \vdash S \isaSet}$, where
  $\Gamma \vdash \Xi \boundary$. A set $S$ is typed with respect to a set
  context $\Xi$ to describe its covariant/contravariant dependence on
  some input objects. Sets are semantically modeled as profunctors.
  
\item Elements/natural transformations: ${\Gamma \pipe \Phi \vdash s :
  R}$, where ${\Gamma \vdash \Phi \isavectx}$ and $\Gamma\pipe
  \underline \Phi \vdash R \isaSet$.  A transformation $s$ has a context
  $\Phi$ of transformation variables and a single output set $R$. To be
  well-formed, the context and set must be parameterized by the same
  contravariant and covariant object variables. To ensure this, we use a
  coercion operation $\underline \Phi$ from transformation contexts to
  set contexts that erases everything in the context but the left-most
  and right-most object variables ($\underline{\alpha : \cat C} = \alpha
  : \cat C$ and $\underline{\Phi} = d^-(\Phi);d^+(\Phi)$).

\item Meta-language types and terms: $\Gamma \vdash A \isaTy$ and
  $\Gamma \vdash M : A$ as in standard dependent type theory.
\end{itemize}

The variable rules for objects and elements are 
\[
\inferrule*{~}{ \Gamma \pipe \alpha : \cat C \vdash \alpha : \cat C}
\qquad
\inferrule*{~}{ \Gamma \pipe \alpha : \cat C, x : R, \beta : \cat D \vdash x : R}
\]
As when using variables in linear logic, the latter rule applies only when the context
contains a single set $R$.  All syntactic forms typed in context admit
an action of substitution.  For types and terms, this is as
usual. Objects $\alpha:\cat C \vdash a : \cat D$ can be substituted for
object variables $\beta : \cat D$ in other objects. We can also
substitute objects into \emph{sets}, that is, if we have a set $P$
parameterized by a contravariant variable $\alpha : \cat C$ and a
covariant variable $\beta: \cat D$, then we can substitute objects $a :
\cat C$ and $b : \cat D$ for these variables $P[a/\alpha;b/\beta]$. This
generalizes the ordinary precomposition of a relation by a function.
Semantically this is the ``restriction'' of a profunctor along two
functors, which is just composition of functors if a profunctor is
viewed as a functor to $\Set$. Modeling this operation as a substitution
considerably simplifies reasoning using profunctors.
Finally we have the action of substitution on elements/natural transformations.
First, we can substitute elements/natural transformations for the set
variables in elements, denoting the composition of natural
transformations.  Second, an element is also parameterized
by a contravariant and a covariant category variable $\alpha;\beta$. We
can think of natural transformations as \emph{polymorphic} in the
categories involved, and so when we make a transformation substitution,
we also \emph{instantiate} the polymorphic category variables with
objects.
The full syntactic details of substitution are included in the appendix.

\subsection{Category Connectives}

In this section we discuss some connectives for constructing categories,
which are specified by introduction and elimination rules in
Figure~\ref{fig:category-connectives} (the $\beta\eta$ equality and
substitution rules are included in the appendix).  The introduction and
elimination rules make use of functors, profunctors, and natural
transformations.
First we introduce the additive connectives: the unit category $1$ and
product category $\cat C \times \cat D$ have the usual introduction
and elimination rules defining functors to/from them.
Next, we introduce the \emph{graph of a profunctor} $\graphProf \alpha
\beta P$. Just as a relation $R : A \times B \to \Set$ can be viewed
as a subset $\{ (a,b) \in A\times B | R(a,b)\}$, any profunctor $P :
{\cat C}_-^o\times {\cat D}_+ \to \Set$ can be viewed as a category with a functor
to ${\cat C}_- \times {\cat D}_+$ (no op), specifically a two-sided discrete
fibration. In set-based category theory, the objects of $\graphProf \alpha
\beta P$ are triples
$(a_-, a_+, s : P(a_-,a_+))$ and morphisms from $(a_-,a_+, s)$ to
$(a_-',a_+', s')$ are pairs of morphisms $f_- : a_- \to a_-'$ and $f_+
: a_+ \to a_+'$ such that $P(\id,f_+)(s) = P(f_-,\id)(s')$.
With various choices of $P$, this connective can be used to define the
arrow category, slice category, comma category and category of
elements.
In our syntax we define it as the universal category $\cat C$ equipped
with functors to $\cat C_-$ and $\cat C_+$ and a natural transformation to $P$.

Lastly, we define the \emph{negative} and \emph{positive} presheaf
categories $\negPresheaf \cat C$ and $\posPresheaf \cat D$. These are
given a syntax suggestive of the fact that they generalize the notion
of a powerset, and so can be thought of as ``power categories''.
Note that we include a restriction that the input category is
\emph{small}, which is an inductively defined by saying all base
categories are small, the unit is small, product of small categories
is small and the graph of a profunctor over small categories is
small. Notably, the presheaf categories themselves are not small.
The negative presheaf category is defined by its universal property
that a functor into it $\cat D \to \negPresheaf \cat C$ is equivalent
to a profunctor $\cat C^o \times \cat D \to \Set$.
The introduction rule constructs an object of the negative presheaf
category from such a profunctor and the elimination rule inverts
it. We use the notation $\negPresheafApp a p$ for the elements of the
induced profunctor. Since $a$ occurs in a negative position, it must
depend only on the contravariant variable $d^-\Xi$ and vice-versa for
$p$.
The positive presheaf category is then the dual. In ordinary
set-theoretic category theory the negative presheaf category is the
usual presheaf category $\Set^{\cat C^o}$, and the positive presheaf
category is the opposite of the dual presheaf category $(\Set^{\cat
  D})^o$.

\begin{figure}[t]
  \begin{scriptsize}
  \[
  \begin{array}{l}
    \text{Unit:}
    \qquad
    \inferrule{ ~ }{\Gamma \vdash 1 \isaCat}
    \qquad
    \inferrule{ ~ }{\Gamma\pipe\alpha:C \vdash () : 1}\\\\
    
    \text{Product:} \quad
    \inferrule
        {\Gamma \vdash \cat C_1 \isaCat \and \cat C_2 \isaCat}
        {\Gamma \vdash {\cat C_1} \times {\cat C_2} \isaCat}
    \quad    
    \inferrule
    {\Gamma \pipe \alpha:\cat C \vdash a_1 : \cat C_1 \and \Gamma \pipe \alpha:\cat C \vdash a_2 : \cat C_2}
    {\Gamma \pipe\alpha:\cat C \vdash (a_1,a_2) : {\cat C_1} \times {\cat C_2}}
    \quad
    \inferrule
    {\Gamma \pipe\alpha:\cat C \vdash a : {\cat C_1} \times {\cat C_2}}
    {\Gamma\pipe\alpha:\cat C\vdash \pi_i a : \cat C_i}
    \\\\
    
    \text{Graph of a profunctor:} \\
    \inferrule
    {\Gamma \pipe \alpha:\cat C;\beta:\cat D \vdash P \isaSet}
    {\Gamma \vdash \graphProf{\alpha}{\beta} P \isaCat}
    \quad
    \inferrule
    {\defaultObCtx \vdash a_- : {\cat C}_-\and
      \defaultObCtx \vdash a_+ : {\cat C}_+\and
      \defaultObCtx \vdash s : P[a_-/\alpha;a_+/\beta]}
    {\defaultObCtx \vdash (a_-,a_+,s) : \graphProf{\alpha:\cat C_-}{\beta:\cat C_+} P}
    \\
    \inferrule
    {\defaultObCtx \vdash a : \graphProf{\alpha:\cat C_-}{\beta} P}
    {\defaultObCtx \vdash \pi_- a : \cat C_-}
    \quad
    \inferrule
    {\defaultObCtx \vdash a : \graphProf{\alpha}{\beta:C_+} P}
    {\defaultObCtx \vdash \pi_+ a : \cat C_+}
    \quad
    \inferrule
    {\defaultObCtx \vdash a : \graphProf{\alpha}{\beta} P}
    {\defaultObCtx \vdash \pi_e a : P[\pi_-a/\alpha;\pi_+a/\beta]}
    \\\\
    \text{Negative Presheaf:} \\
    \inferrule
    {\Gamma \vdash \cat C \isaCat \and \cat C \isSmall}
    {\Gamma \vdash \negPresheaf {\cat C} \isaCat}
    \quad
    \inferrule
    {\Gamma \pipe d^-\Xi \vdash a : \cat C \and
     \Gamma \pipe d^+\Xi \vdash p : \negPresheaf {\cat C}
    }
    {\Gamma \pipe \Xi \vdash \negPresheafAppPtoX p a \isaSet}
    \quad
    \inferrule
    {\Gamma \pipe \alpha:\cat C;\beta:\cat D \vdash R : \isaSet}
    {\Gamma \pipe \beta:\cat D \vdash \lambda \alpha:\cat C. R : \negPresheaf {\cat C}}
    \\
    \\
    \text{Positive Presehaf:} \\
    \inferrule
    {\Gamma \vdash \cat D \isaCat \and \cat D \isSmall}
    {\Gamma \vdash \posPresheaf {\cat D} \isaCat}
    \quad
    \inferrule
    {\Gamma \pipe d^-\Xi \vdash p : \posPresheaf {\cat D}\and
     \Gamma \pipe d^+\Xi \vdash a : \cat D
    }
    {\Gamma \pipe \Xi \vdash \posPresheafAppPtoX p a \isaSet}
    \quad
    \inferrule
    {\Gamma \pipe \alpha:\cat C;\beta:\cat D \vdash R : \isaSet}
    {\Gamma \pipe \alpha:\cat C \vdash \lambda \beta:\cat D. R : \posPresheaf {\cat D}}
  \end{array}
  \]
  \end{scriptsize}
  \caption{Category Conectives}
  \label{fig:category-connectives}
\end{figure}

\subsection{Set Connectives}

Next, in Figure~\ref{fig:set-connectives}, we cover the connectives for
the sets/profunctors, which classify elements/natural transformations
(the $\beta/\eta$-rules are in the appendix).
First, the unit set $\punitinXfromYtoZ {\cat C} a b$ is our syntax for
the profunctor of morphisms in $\cat C$ instantiated at generalized
objects $a$ and $b$.
Its introduction and elimination rules are analogous to the usual
rules for equality in intensional Martin-L\"of type theory.  The
introduction rule is the identity morphism (reflexivity) and the
elimination rule is an induction principle: we can use a term of $s :
\punitinXfromYtoZ {\cat C} a b$ by specifying the behavior when $s$ is
of the form $\punitrefl \alpha$ in the form of a continuation
$\alpha. t$.  Like the $J$ elimination rule for equality in
Martin-L\"of type theory, $P$ must be ``fully general'',
i.e. well-typed for variables $\alpha$ and $\beta$. This is because
for distinct variables $\alpha$ and $\beta$, $\punitinXfromYtoZ{\cat
  C}{\alpha}{\beta}$ denotes the unit in a virtual double category,
which has a universal property, but $\punitinXfromYtoZ{\cat C} a b$
denotes a restriction of the unit, which in general does not.  Those
familiar with linear logic as in e.g.~\cite{polakow-pfenning} might
expect a more general rule, where the continuation $t$ is allowed to
use variables that are not used in $s$, i.e., have a context $\Phi_l
\jnctx \Phi_r$ and the conclusion of the rule to have a context
$\Phi_l \jnctx \Phi \jnctx \Phi_r$. Because of dependency, this is not
necessarily well-formed in cases where the endpoints $a$ and $b$ of
$\punitinXfromYtoZ{\cat C} a b$ are not distinct variables.  However,
the instances of this more general rule that do type check are
derivable from our more restricted rule using right/left-hom types.

The tensor product of sets is a kind of combined existential quantifier
and monoidal product, which we combine into a single notation
$\tensorexistsXwithYandZ \beta P Q$, where $\beta$ is the covariant
variable of $P$ and the contravariant variable of $Q$. Then the
covariant variable of the tensor product is the covariant variable of
$Q$ and the contravariant variable similarly comes from $P$.
In ordinary category theory, this is the \emph{composition} of
profunctors, and is defined by a coend of a product. We require that
the variable $\beta$ quantifies over a small category $\cat D$, as in
general this composite doesn't exist for large categories.
The introduction and elimination are like those for a combined tensor
product and existential type: the introduction rule is a pair of
terms, with an appropriate instantiation of $\beta$, and the
elimination rule says to use a term of a tensor product, it is
sufficient to specify the behavior on two elements typed with an
arbitrary middle object $\beta$.

Next, we introduce the contravariant ($\homlallXYtoZ \alpha R P$) and
covariant ($\homrallXYtoZ \alpha R P$) homs of sets, which are
different from each other because we are in an ordered logic.
These are a kind of universally quantified
function type, where the universally quantified variable must occur
with the same variance in domain and codomain. In the contravariant
case, it occurs as the contravariant variable in both, and vice-versa
for the covariant case.
To highlight this, the notation for the contravariant dependence puts
the quantified variable on the \emph{left} of the triangle, as
contravariant variables occur to the left of the covariant variable,
and similarly the covariant hom has the quantified variable on the
right.
Similar to ordered lambda calculus, the covariant hom is
right-associative while the contravariant hom is left-associative.
Then the
covariant variable of the contravariant hom set is the covariant
variable of the codomain and, and the contravariant variable of the
hom set is the \emph{covariant} variable of the domain, as the two
contravariances cancel. The covariant hom is dual.
Semantically, in ordinary category theory these are known as the
\emph{hom} of profunctors and are adjoint to the composition of
profunctors \cite{benabou2000distributors}.
The two connectives have similar introduction and elimination rules in
the form of $\lambda$ terms abstracting over both the object of the
category and the element of the set, and appropriate application
forms.
To keep with our invariant that the variable occurrences occur left to
right in the term syntax in a manner matching the context, we write
the covariant application in the usual order $\homrappXtoYatZ s t a$
where the function is on the left and the argument is on the right,
and the contravariant application in the flipped order.
We also write the instantiating object as a superscript to
de-emphasize it, as in practice it can often be inferred.

Finally, we have the cartesian unit and product sets, which are
analogous to the normal unit and product of types. The most notable
point to emphasize is that in the formation rule for the product, the
two subformulae should have the same covariant and contravariant
dependence (as with linear logic, some constructions can
syntactically use a variable more than once and still be ``linear'').  

\begin{figure}[t]
  \begin{scriptsize}
  \[
  \begin{array}{l}
    
    \text{Unit/morphism set:}\\
    \inferrule
  {\Gamma \pipe d^-\Xi \vdash a_1 : \cat C\\\\ \Gamma \pipe d^+ \Xi \vdash a_2 : \cat C}
  {\Gamma \pipe \Xi \vdash \punitinXfromYtoZ {\cat C} {a_1} {a_2} \isaSet}
  \quad
    \inferrule
    {\Gamma\pipe\beta: \cat D \vdash a : \cat C}
    {\Gamma \pipe \beta : \cat D \vdash \punitrefl a : \punitinXfromYtoZ {\cat C} a a}
    \quad
    \inferrule
    {\Gamma\pipe \alpha:\cat C; \beta:\cat C \vdash P \isaSet \\\\
     \Gamma\pipe \alpha:\cat C\vdash t : P[\alpha/\alpha;\alpha/\beta]\\\\
     \Gamma\pipe\Phi \vdash s : \punitinXfromYtoZ {\cat C} a b}
    {\Gamma \pipe \Phi \vdash \punitelimkontZatABC {\alpha. t} A s B : P[a/\alpha;b/\beta]}
    \\


    \\
    \text{Tensor product:}\\
    \inferrule
    {\cat D \isSmall\\\\
      \Gamma \pipe d^-\Xi; \beta:{\cat D}\vdash P \isaSet\\\\
     \Gamma \pipe \beta:{\cat D};d^+\Xi \vdash Q \isaSet}
    {\Gamma \pipe \Xi \vdash \tensorexistsXwithYandZ {\beta:{\cat D}} P Q \isaSet}
    \quad
  \inferrule
  {\Gamma \pipe d^+\Psi_s \vdash b : \cat D\\\\
   \Gamma \pipe \Psi_s \vdash s : P[b/\beta] \\\\
   \Gamma \pipe \Psi_t \vdash t : Q[b/\beta]}
  {\Gamma\pipe\Psi_s \jnctx \Psi_t\vdash \tensorintroatXwithYandZ b s t : \tensorexistsXwithYandZ {\beta : \cat D} P Q}
  \quad
  \inferrule
  {
   \Gamma\pipe \Phi_l \jnctx x:P, \beta:\cat D, y:Q \jnctx \Phi_r \vdash t : R\\\\
   \Gamma \pipe\Phi_m \vdash s : \tensorexistsXwithYandZ {\beta:\cat D} P Q}
  {\Gamma\pipe\Phi_l\jnctx \Phi_m \jnctx \Phi_r \vdash \tensorelimWkontZ s {x,\beta,y. t} : R}
  \\



  \\
  \text{Right hom:} \quad
  \inferrule
  {d^+\Xi \isSmall\\\\ \Gamma \pipe d^+\Xi; \alpha : \cat C \vdash  R
    \isaSet \\\\ \Gamma \pipe  d^-\Xi; \alpha : \cat C \vdash P \isaSet}
  {\Gamma \pipe \Xi \vdash \homrallXYtoZ {\alpha:\cat C} R P \isaSet}
  \quad
  \inferrule
  {\Gamma \pipe \Phi, x : R, \alpha : \cat C \vdash t : P}
  {\Gamma \pipe \Phi \vdash
   \homrlambdaXatYdotZ {x : R} {\alpha : \cat C} t : \homrallXYtoZ {\alpha:\cat C} R P}
  \quad
  \inferrule
  {\Gamma\pipe\Phi_f \vdash s : \homrallXYtoZ {\alpha:\cat C} R P \\\\
    d^+\Phi_a \vdash a : \cat C\\\\
    \Phi_a \vdash t : R[a/\alpha]
  }
  {\Gamma\pipe \Phi_f \jnctx \Phi_a \vdash \homrappXtoYatZ s t a : P[a/\alpha]}\\
  \quad


  \\ 
  \text{Left hom:} \quad
  \inferrule
    {d^-\Xi \isSmall\\\\
      \Gamma \pipe \alpha:\cat C; d^-\Xi\vdash R \isaSet \\\\ \Gamma
      \pipe \alpha : \cat C; d^+\Xi\vdash P \isaSet }
    {\Gamma \pipe \Xi \vdash \homlallXYtoZ {\alpha: \cat C} R P \isaSet}
  \quad
  \inferrule
  {\Gamma\pipe\alpha:\cat C, x : R, \Phi \vdash t : P}
  {\Gamma\pipe\Phi \vdash \homllambdaXatYdotZ {x : R} {\alpha : \cat C} t : \homlallXYtoZ {\alpha:\cat C} R P}
  \quad
  \inferrule
  {\Gamma \pipe d^-\Phi_a \vdash a : \cat C\\\\
   \Gamma \pipe \Phi_a \vdash s : R[a/\alpha] \\\\
   \Gamma \pipe \Phi_f \vdash t : \homlallXYtoZ {\alpha:\cat C} R P}
  {\Gamma \pipe \Phi_a \jnctx \Phi_f \vdash \homlappXtoYatZ t s a : P[a/\alpha]}\\
  %

  \\
  \text{Cartesian unit and products:} \quad
      \inferrule{~}{\Gamma\pipe \Xi \vdash 1 \isaSet}
      \quad
      \inferrule{~}{\Gamma\pipe \Phi \vdash () :  1}
      \\ 
    \inferrule{\Gamma\pipe\Xi \vdash R \isaSet \\\\ \Gamma\pipe\Xi \vdash S\isaSet}
              {\Gamma\pipe\Xi \vdash R \times S\isaSet}
    \quad
    \inferrule{\forall i \in \{1,2\}.~ \Gamma \pipe\Phi \vdash s_i : R_i}{\Gamma \pipe\Phi \vdash (s_1,s_2) : R_1 \times R_2}
    \quad
    \inferrule{\Gamma \pipe\Phi\vdash s : R_1 \times R_2}{\Gamma \pipe\Phi \vdash \pi_i s : R_i}
  \end{array}
  \]
  \end{scriptsize}
  \caption{Set Connectives}
  \label{fig:set-connectives}
\end{figure}

\subsection{Type Connectives}

Finally, we briefly describe the connectives for the ``meta-logic'',
which extends Martin-L\"of type theory with $\Pi$/$\Sigma$ and extensional
identity types (with their standard rules). We use
extensional identity types so that the description of models is simpler,
but intensional identity types could be used instead.
The types we include are \emph{universes} for the object categorical
logic: types of small categories and locally small categories, functors,
profunctors and natural transformations.
The rule for the types of small categories and (large) categories are
very similar: any definable category defines an element of type $\Cats$,
and any element of that type can be reflected back into a category. The
only difference for $\smallCats$ is that the categories involved
additionally satisfy $\cat C \isSmall$. Again we elide the $\beta\eta$
principles, which state that $\quoth{-}$ and $\unquoth{-}$ are mutually
inverse.
Since every small category $\cat C \isSmall$ is a category $\cat C
\isaCat$, there is a definable inclusion function from $\smallCats$ to
$\Cat$ and the $\beta\eta$ properties ensure that this is a
monomorphism.

Next, we have the types of all functors and profunctors between any
two fixed categories.
The introduction and elimination forms are those for unary and binary
function types respectively, where metalanguage terms of type
$\varr{\cat C}{\cat D}$ can be used to construct an object/functor,
while metalanguage terms of type $\harr{\cat C}{\cat D}$ can be used to
construct a set/profunctor.

Finally we include a type $\pendallXdotY {\alpha:\cat C} P$ which we
call the set of ``natural elements'' of $P$.
The name comes from the case that $P$ is of the form
$\punitinXfromYtoZ {}{F(\alpha)}{G(\alpha)}$ in which case the type
$\pendallXdotY {\alpha:\cat C} \punitinXfromYtoZ
{}{F(\alpha)}{G(\alpha)}$ can be interpreted as the set of all natural
transformations from $F$ to $G$.
More generally this is modeled as an end, and we notate it with a
universal quantifier (just as we do for the quantifiers in left/right
hom types).  Syntactically, $\pendallXdotY{\alpha}{P}$ is a
meta-language type that represents elements/natural transformations with
exactly one free variable.  

\begin{figure}
  \begin{scriptsize}
  \[
  \begin{array}{l}
    \inferrule
    {~}
    {\Gamma\vdash \smallCats}
    \quad
    \inferrule
    {\Gamma \vdash \cat C \isSmall}
    {\Gamma \vdash \quoth {\cat C} : \smallCats}
    \quad
    \inferrule
    {\Gamma \vdash M : \smallCats}
    {\Gamma \vdash \unquoth M \isSmall}
    \quad
    \inferrule
    {~}
    {\Gamma\vdash \Cats}
    \quad
    \inferrule
    {\Gamma \vdash \cat C \isaCat}
    {\Gamma \vdash \quoth {\cat C} : \Cat}
    \quad
    \inferrule
    {\Gamma \vdash M : \Cat}
    {\Gamma \vdash \unquoth M \isaCat}
    \\ \\
    
    \inferrule
    {\Gamma \vdash \cat C \isaCat \and\Gamma \vdash \cat D \isaCat}
    {\Gamma \vdash \varr{\cat C}{\cat D} \isaTy}
    \quad
    \inferrule
    {\Gamma \pipe \alpha:\cat C \vdash A : \cat D}
    {\Gamma \vdash \lambda \alpha:\cat C. A : \varr{\cat C}{\cat D}}
    \quad
    \inferrule
    {\Gamma \pipe \alpha:\cat C \vdash A : \cat D\and
      \Gamma \vdash M : \varr{\cat D}{\cat E}}
    {\Gamma\pipe \alpha:\cat C \vdash M A : \cat E}
    \\ \\
    
    \inferrule
    {\Gamma \vdash \cat C \isaCat \and\Gamma \vdash \cat D \isaCat}
    {\Gamma \vdash \harr{\cat C}{\cat D} \isaTy}
    \quad
    \inferrule
    {\Gamma \pipe \alpha:\cat C;\beta:\cat D \vdash R \,\isaSet}
    {\Gamma \vdash \lambda \alpha:\cat C;\beta:\cat D. R : \harr{\cat C}{\cat D}}
    \quad
    \inferrule
    {\Gamma \vdash M : \harr{\cat C}{\cat D}\\\\
     \Gamma \pipe d^- \Xi \vdash A : \cat C\\\\
     \Gamma \pipe d^+ \Xi \vdash B : \cat C}
    {\Gamma \pipe \Xi \vdash M A\, B \isaSet}
    \\ \\
    
    \inferrule
    {\Gamma \pipe \alpha:\cat C\vdash P \isaSet}
    {\Gamma \vdash \pendallXdotY {\alpha:{\cat C}} P \isaTy}
    \quad
    \inferrule
    {\Gamma \pipe \alpha:\cat C\vdash t : P}
    {\Gamma \vdash \pendlambdaXdotY \alpha t : \pendallXdotY \alpha P}
    \quad
    \inferrule
    {\Gamma \vdash M : \pendallXdotY \alpha P\and
     \Gamma \pipe \beta:\cat D \vdash a : \cat C}
    {\Gamma \pipe \beta:\cat D \vdash \pendappXtoY M a : P[a/\alpha]}
  \end{array}
  \]
  \end{scriptsize}
  \caption{Type Connectives}
\end{figure}

\section{Formal Category Theory in VETT}
\label{sec:examples}

To demonstrate what formal category theory in VETT looks like, we
demonstrate some basic definitions and theorems.  While it is well known
that much category theory can be formalized in virtual equipments, we
show these examples to demonstrate how the VETT syntax gives a more
familiar syntax to these constructions, while still avoiding the need
for explicit naturality and functoriality side conditions.  We have
mechanized some of the results in this section (e.g. Lemma~\ref{lem:yoneda}
and Lemma~\ref{lem:fubini-short} and the maps in
Lemma~\ref{lem:adjunction}) in
Agda.\footnote{\url{https://github.com/maxsnew/virtual-equipments/blob/master/agda/Examples.agda}}

First, we using the elimination for the unit set, we can see that all
constructions are (pro-)functorial:
\begin{construction}
  \label{construction:synthetic-composition}
  For any small category $\cat C$, we can construct natural elements
  \begin{enumerate}
  \item Identity: $\pendallXdotY {\alpha:\cat C} {\punitinXfromYtoZ {\cat C}{\alpha}{\alpha}}$
  \item Composition: $\hombinary {\alpha_1:\cat C} {(\punitinXfromYtoZ {\cat C}{\alpha_1}{\alpha_2})} {\alpha_2:\cat C} {(\punitinXfromYtoZ {\cat C}{\alpha_2}{\alpha_3})} {\alpha_3:\cat C} {(\punitinXfromYtoZ {\cat C}{\alpha_1}{\alpha_3})}$
  \item Functoriality: for any $F : \varr {\cat C}{\cat D}$, $\homunary {\alpha_1:\cat C} {(\punitinXfromYtoZ {\cat C}{\alpha_1}{\alpha_2})} {\alpha_2:\cat C} {(\punitinXfromYtoZ {\cat D} {F(\alpha_1)}{F(\alpha_2)})}$.
   \item Profunctoriality: for any $R : \harr {\cat C}{\cat D}$ if
     $\cat D$ is small then \\ $\homtrinary {\alpha_1:\cat C}
     {(\punitinXfromYtoZ{\cat C} {\alpha_1}{\alpha_2})} {\alpha_2:\cat C}
     {R{\alpha_2}{\beta_2}} {\beta_2:\cat D} {(\punitinXfromYtoZ{\cat D} {\beta_2}
       {\beta_1})} {\beta_1:\cat D} {R\alpha_1\beta_1}$
  \end{enumerate}
\end{construction}
Identity and Composition generalize the reflexivity and transitivity
properties of equality, respectively, with the lack of symmetry being
a key feature of the generalization.
In addition, we can prove that the (pro)-functoriality axioms commute
with the composition proof by the $\eta$ principle for the unit.
(Pro-)Functoriality generalizes the statement that all functions and
relations respect equality.
Naturality is more complex to state, and it is a statement about the
\emph{proofs} so it has no analog in ordinary higher-order
logic. The following version is stated for any \emph{profunctor}, with
the usual case of naturality arising when $R \alpha \beta =
\punitinXfromYtoZ {\cat C}{F \alpha}{G \beta}$.
\begin{lemma}[Naturality]
  For any $t : \pendallXdotY {\alpha:\cat C}{\harrapp R \alpha
    \alpha}$, by composing with profunctoriality, we can construct
  terms $\alpha_1:{\cat C}, f : \punitinXfromYtoZ {\cat C} {\alpha_1}
  {\alpha_2}, \alpha_2:{\cat C} \vdash \textrm{lcomp}(f,\pendappXtoY t
  {\alpha_2})$ and $\textrm{rcomp}(\pendappXtoY t {\alpha_1}, f) :
  \harrapp R {\alpha_1}{\alpha_2}$ that are both equal to
  $\punitelimtoYwithkontZ f {t}$.
\end{lemma}

Next, we turn to some of the central theorems of category theory, the
Yoneda and Co-Yoneda lemmas. Despite being ultimately quite
elementary, these are notoriously abstract. In \ohol{}, we view these
as ordered generalizations of some very simple tautologies about
equality. For instance, the Yoneda lemma generalizes the equivalence
between the formulae $\forall y. x = y \Rightarrow P y$ and $P x$ for
any $x$.
\begin{lemma} \label{lem:yoneda}
  Let $\alpha : \cat C$ and $\pi : \posPresheaf {\cat C}$. Then
  \begin{enumerate}
  \item (Yoneda) The profunctor 
    $\homrallXYtoZ {\alpha'} {(\punitinXfromYtoZ {\cat C} {\alpha}{\alpha'})} {(\posPresheafAppPtoX \pi {\alpha'})}$ is isomorphic to $\posPresheafAppPtoX \pi \alpha$
  \item (Co-Yoneda) The profunctor $\tensorexistsXwithYandZ {\alpha'}{(\posPresheafAppPtoX \pi \alpha')} {(\punitinXfromYtoZ{}{\alpha'}{\alpha})}$ is isomorphic to $\posPresheafAppPtoX \pi \alpha$
  \end{enumerate}
\end{lemma}
The proofs both follow from the unit elimination rule, which is
essentially the Yoneda lemma---the two cases of showing (1) is an
isomorphism are precisely the $\beta$ and $\eta$ rules for the unit.

Next, we have the ``Fubini'' theorems, which relate the tensor and hom
types. The statement and proofs for these theorems are analogous to
proofs relating tensor and hom in ordered logic. For instance, the
second isomorphism below is analogous to the equivalence $(P \odot Q)
\multimap R \cong P \multimap Q \multimap R$ in ordered logic.
\begin{lemma}[{Fubini}]
  \label{lem:fubini-short}
  The following isomorphisms hold when the corresponding profunctors
  are well typed.
  \begin{enumerate}
  \item $\tensorexistsXwithYandZ {\beta} {\harrapp P \alpha \beta} {(\tensorexistsXwithYandZ \gamma {\harrapp Q \beta \gamma} {\harrapp R \gamma \delta})} \cong \tensorexistsXwithYandZ {\gamma} {(\tensorexistsXwithYandZ \beta {\harrapp P \alpha \beta} {\harrapp Q \beta \gamma})} {\harrapp R \gamma \delta}$
  \item $\homrallXYtoZ {\gamma} {(\tensorexistsXwithYandZ \beta {\harrapp P \delta \beta} {\harrapp Q \beta \gamma})} {\harrapp S \alpha \gamma} \cong \homrallXYtoZ \beta {\harrapp P \delta \beta} {\homrallXYtoZ \gamma {\harrapp Q \beta \gamma} {\harrapp S \alpha \gamma}}$
  \item $\homlallXYtoZ {\gamma} {(\tensorexistsXwithYandZ \beta {\harrapp P \gamma \beta} {\harrapp Q \beta \alpha})} {\harrapp S \gamma \delta} \cong \homlallXYtoZ \beta {\harrapp Q \beta \alpha} {\homlallXYtoZ \gamma {\harrapp P \gamma \beta} {\harrapp S \gamma \delta}}$
  \item $\homrallXYtoZ \gamma {\harrapp Q \delta \gamma} {(\homlallXYtoZ \beta {\harrapp P \beta \alpha} {\harrapp S \beta \gamma})} \cong \homlallXYtoZ \beta {\harrapp P \beta \alpha} {(\homrallXYtoZ \gamma {\harrapp Q \delta \gamma} {\harrapp S \beta \gamma})}$
  \item $\pendallXdotY \alpha \homrallXYtoZ \beta {\harrapp P \alpha \beta} {\harrapp Q \alpha \beta} \cong \pendallXdotY \beta \homlallXYtoZ \alpha {\harrapp P \alpha \beta} {\harrapp Q \alpha \beta}$
  \end{enumerate}  
\end{lemma}
\begin{proof}
  We show one case as an example, the forward direction of (1) is
  given by $\lambdaunary \alpha {x} \delta {\tensorelimWkontZ x
    {p,\beta,y. \tensorelimWkontZ y
      {q,\gamma,r. \tensorintroatXwithYandZ {\gamma}
        {\tensorintroatXwithYandZ \beta p q} {r}}}}$
\end{proof}

Next, we can prove that two definitions of an adjunction are
equivalent:
\begin{lemma} \label{lem:adjunction}
  For $R : \varr {\cat D}{\cat C}$ and $L : \varr {\cat
    C}{\cat D}$, the following are in bijection:
  \begin{enumerate}
  \item An isomorphism of profunctors $(\punitinXfromYtoZ {\cat D} {L \alpha}{\beta}) \cong (\punitinXfromYtoZ {\cat C} \alpha {R \beta})$
  \item A unit $\eta : \pendallXdotY \alpha {\punitinXfromYtoZ {\cat C} \alpha
    {R(L\alpha)}}$ and co-unit $\varepsilon: \pendallXdotY \beta
    {\punitinXfromYtoZ {\cat D} {L(R(\beta))}{\beta}}$ satisfying triangle
    identities.
  \end{enumerate}
\end{lemma}
\begin{proof}
  Given the forward homomorphism $\text{lr}$, we can construct $\eta =
  \pendlambdaXdotY \alpha \appunary {\text{lr}} \alpha {\punitrefl
    \alpha} {L \alpha}$.  Given the unit we can reconstruct the
  forward homomorphism using
  $\text{comp}$ (composition) and $\text{fctor}$ (functoriality) from
  Construction~\ref{construction:synthetic-composition} as\\
  $\appbinary {\text{comp}} \alpha
  {\pendappXtoY \eta \alpha} {R(L\alpha)} {(\appunary
    {\text{fctor}(R)}{L \alpha} {f} {\beta})} {R \beta}$.
\end{proof}

We can define weighted limits, which as special cases include ordinary
limits and Kan extensions.
\begin{definition}
  For a functor $D : \varr {\cat J}{\cat C}$ and a profunctor $W :
  \harr{\cat K}{\cat J}$, the limit of $D$ weighted by $W$ is (if it
  exists) a functor $\weightedLimitDW D W : \varr{\cat K}{\cat C}$
  with an isomorphism $\punitinXfromYtoZ {\cat C} {\alpha} {(\weightedLimitDW D W) k} \cong \homrallXYtoZ j {W k j} {(\punitinXfromYtoZ {\cat C} \alpha {D j})}$
\end{definition}
This generalizes the usual definition that a morphism into a limit is
a cone over the diagram $(\punitinXfromYtoZ {\cat C} \alpha {D j})$ to
be parameterized by a weight $W k j$.
Then we can prove the well-known theorem that right adjoints preserve (weighted) limits:
\begin{theorem}
  If $\weightedLimitDW D W$ exists and is a limit and $R : \varr{\cat
    C}{\cat C'}$ has a left adjoint $L$, then $\lambda
  \kappa. R((\weightedLimitDW D W) \kappa)$ is the limit of $\lambda
  j. R(D j)$ weighted by $W$.
\end{theorem}
\begin{proof}
\[
    \punitinXfromYtoZ {} {\gamma} {R ((\weightedLimitDW D W) \kappa)}
    \cong \punitinXfromYtoZ {}{L \gamma} {(\weightedLimitDW D W) \kappa}
    \cong \homrallXYtoZ {j}{W k j} {\punitinXfromYtoZ {}{L \gamma}{D j}}
    \cong \homrallXYtoZ {j}{W k j} {\punitinXfromYtoZ {}{\gamma}{R(D j)}}\\
\]
This is a high level proof in terms of isomorphisms that may be written
in VETT. The first two steps are the instantiation of assumptions
(adjointness, weighted limits). The last step uses the fact that a
natural isomorphisms lift to natural isomorphism of homs of profunctors.
The construction of this isomorphism illustrates how naturality need not
be proved explicitly in \vett{}.  For any $\phi: \homunary \alpha {R'
  \alpha \beta} \beta {R \alpha \beta}$ and $\psi : \homunary \gamma {S
  \gamma \beta} \beta {S' \gamma \beta}$ we can construct a natural
transformation $\phi \triangleright \psi : \hombinary \gamma
{(\homrallXYtoZ \beta {R \alpha \beta}{S \gamma \beta})} {\alpha}
{R'\alpha\beta} \beta {S'\gamma\beta}$ as \\ $\lambdabinary {\gamma} {f}
{\alpha} r \beta {\appunary \psi \gamma {(\homrappXtoYatZ f {(\appunary
      \phi \alpha r \beta)} \beta)} \beta}$. Furthermore if $\phi$ and
$\psi$ have inverses, then $\phi^{-1} \triangleright \psi^{-1}$ is the
inverse of $\phi \triangleright \psi$.
\end{proof}

\section{Semantics}
\label{sec:semantics}

Next, we develop the basics of the model theory for \ohol{}.
First, we define a sound and complete notion of categorical model
based on hyperdoctrines of virtual equipments.
Then we instantiate this general notion of model to show that the
\ohol{} can be interpreted in ordinary category theory as well as
enriched, internal and indexed notions.


First, we can model the judgmental structure of the unary type theory
and predicate logic in \emph{virtual double categories} that are
\emph{split fibrant} and have a notion of \emph{small object}
\cite{LEINSTER2002391,Cruttwell2010}. We briefly recount the structure
present in a virtual double category, but see \citet{Cruttwell2010}
for a precise definition of the composition rules for 2-cells and
functor of virtual double categories.

\begin{definition}
  A virtual double category $\mathcal V$ consists of
  \begin{enumerate}
  \item A category $V_o$ of ``objects and vertical arrows''
  \item A set ${\mathcal V}_h$ of ``horizontal arrows'' with source
    and target functions $s,t : {\mathcal V}_h \to {\mathcal V_o}^2$
  \item Sets of 2-cells of the following form, 
    with appropriate ``multi-categorical'' notions of identity and
    composition:
\[\begin{tikzcd}[row sep=tiny]
	{C_0} & \cdots & {C_{n}} \\
	& \phi \\
	{D_0} && {D_1}
	\arrow["{R_n}", "\shortmid"{marking}, from=1-2, to=1-3]
	\arrow["f"', from=1-1, to=3-1]
	\arrow["g", from=1-3, to=3-3]
	\arrow["S"', "\shortmid"{marking}, from=3-1, to=3-3]
	\arrow["{R_0}", "\shortmid"{marking}, from=1-1, to=1-2]
\end{tikzcd}\]
    We say that the 2-cell $\phi$ has $S$ as codomain, the sequence
    $R_0 \ldots R_n$ as domain and call $f$ and $g$ the left and right
    ``frames'', or that $\phi$ is framed by $f$ and $g$.
  \end{enumerate}
  We say a virtual double category is \emph{split fibrant} when it has
  a choice of \emph{restrictions}, that is, for any horizontal arrow
  $R : C \pto D$ and vertical arrows $f : C' \to C$ and $g : D' \to D$
  there is a chosen horizontal arrow $R(f,g) : C' \pto D'$ with a
  cartesian 2-cell to $R$ framed by $f,g$ and these chosen cartesian
  lifts are functorial in $f,g$ (\cite{Shulman08}).  A \emph{choice of
  small objects} is a subset of the objects $V_s \subseteq V_o$.  A
  \emph{morphism} of split fibrant virtual double categories with
  small objects is a functor of the virtual double categories that
  additionally preserves the restrictions and smallness of objects.
  This defines a category $\fVDC$.
\end{definition}

In the presence of restrictions, every 2-cell can be represented as a
``globular'' 2-cell where the left and right frame are identities
\cite{Shulman08}. For example the 2-cell $\phi$ above can be
represented as one with the same domain but whose codomain is
$S(f,g)$. This property is crucial for the completeness of our
semantics as we only include a syntax for these globular terms (proof
of Construction~\ref{cons:syn-model}).
Each component of this definition has a direct correspondence to a
syntactic structure in \ohol{}. The objects of $\mathcal V_o$ models
the category judgment and the morphisms model the functor
judgment. The set $\mathcal V_h$ models the profunctor judgment. A
composable string $R_0 \cdots R_n$ models the profunctor contexts. The
2-cells correspond to the natural transformation judgment where we
have taken the restriction $S(F,G)$ of the codomain.
Note that Cruttwell and Shulman define a \emph{virtual equipment} to
be a virtual double category with all restrictions and all units. The
units are the model of the unit of profunctors connective and so all
of our models with the unit will be virtual equipments, hence the name
VETT.

To model the dependent type theory and indexing of category-theoretic
judgments by a $\Gamma$ with an action of substitution, we use a
variation on Lawvere's notion of \emph{hyperdoctrine} for modeling
predicate logic\cite{Lawvere69}\footnote{note that unlike in
hyperdoctrines, we do not require \emph{quantifiers} adjoint to
substitution}:
\begin{definition}[VETT Judgmental model]
  A VETT judgmental model ($\VMJ$) is a pair of a category with
  families $\mathcal C$ and a functor $V^{(-)} : \mathcal C^o \to
  \fVDC$.
\end{definition}
Categories with families $\mathcal C$ model dependent type
theory~\cite{dybjer-cwf} and for each semantic context $\Gamma$,
$V^\Gamma$ models the VETT judgments in context $\Gamma$, with the
functoriality modeling the fact that all of these judgments admit a
well-behaved action of substitution.
A $\VMJ$ is then precisely the structure corresponding to the
judgments and actions of substitution in VETT.
\begin{construction}[Syntactic Model]
  \label{cons:syn-model}
  The syntax of VETT with with any subset of connectives are included
  presents a $\VMJ$.
\end{construction}
\begin{proof}
  Define the category of families using the dependent
  type structure and the virtual equipment structure having
  ($\alpha$-equivalence classes of) syntactic categories as objects,
  functors/sets as vertical/horizontal arrows and interpreting
  compositions/restrictions as substitutions.  The biggest gap between syntax and semantics is in the
  definition of the 2-cells.  A 2-cell from\\ $(\alpha_1:\cat
  C_1\mathord{;}\alpha_2:\cat C_2\vdash R_1),(\alpha_2:\cat
  C_2;\alpha_3:\cat C_3 \vdash R_2),\ldots$ to $(\beta_1:\cat
  D_1;\beta_2 : \cat D_2\vdash S)$ with frames $\alpha_1:\cat C_1\vdash b_1 : \cat
  D_1$ and $\alpha_n:\cat C_n\vdash b_2 : \cat D_2$ is given by a term
  $x_1:R_1,x_2:R_2\ldots \vdash s :
  S[b_1/\beta_1;b_2/\beta_2]$. Composition is defined by substitution.
\end{proof}

Then the \emph{connectives} of VETT each precisely correspond to a
universal construction in a $\VMJ$. The $\Pi,\Sigma,\text{Id}$ types
correspond to their standard semantics in a CwF and the connectives
for categories and profunctors correspond to universal constructions
in the virtual double categories. Products of categories are
interpreted as products in the vertical category, and products of sets
as products in the category of pro-arrows and 2-cells. The units,
tensor and covariant and contravariant homs are modeled by the
universal properties of the same names, as described in
\citet{Shulman08}. The graph of a profunctor is modeled by
\emph{tabulators} \citet{grandis-pare99}. Finally, the covariant and
contravariant presheaf categories can be described as a weakening of
the definition of a Yoneda equipment from \citet{diliberti-loregian}
to virtual double categories. More detailed descriptions of these
universal properties are included in the extended version
\cite{extended-version}.
Then the soundness and completeness of this notion of categorical
model is formalized by the following initiality theorem.
\begin{theorem}[Initiality]
  The syntax of VETT with any subset of connectives that includes the
  hom types presents a $\VMJ$ that is initial in the category of
  $\VMJ$ with the chosen instances of the universal properties and
  functors that preserve such chosen instances.
\end{theorem}
\begin{proof}
  The construction \ref{cons:syn-model} can be extended for any
  connective modularly, with the exception that the unit relies on the
  presence of hom sets in order to satisfy the ``distributivity''
  requirement that its elimination can occur in any context. Then we
  can construct the unique morphism to any HVE induction on syntax.
\end{proof}


Now that we have a category-theoretic notion of model, we give some
model construction theorems that can be used to justify our intuitive
notion of semantics in (enriched, internal, indexed) category theory.
First, we can extend any set-theoretic model of the category theoretic
judgments to a hyperdoctrine of models where the category of families
is the category of sets:
\begin{construction}
  Given a $\mathcal V \in \fVDC$, we can construct a $\VMJ$ $\mathcal
  V^{-}: \Set \to \text{vDbl}_r$ by defining of $(\mathcal
  V^\Gamma)_o$ to be functions $\mathcal V_o^\Gamma$, and similarly
  for morphisms and 2-cells with all operations given pointwise.
\end{construction}

%

Then to define a model of VETT with a collection of connectives it is
sufficient to construct a virtual equipment with the corresponding
universal properties.
The ``standard model'' is the virtual double category of locally
small categories where the small objects are the small categories.
\begin{construction}
  Fix a cardinal $\kappa$. The virtual double category $\Cat_{\kappa}$
  is defined to have as objects locally $\kappa$-small categories,
  small objects as $\kappa$-small categories, vertical morphisms as
  functors, horizontal arrows as functors $\cat C^o \times \cat D \to
  \kappa Set$ and 2-cells as morphisms of profunctors. Restriction of
  profunctors is given by composition, which is strictly associative
  and unital.  $\Cat_{U}$ has objects satisfying the universal
  properties of all connectives in VETT.
\end{construction}

More generally, categories internal to, enriched in and/or indexed by
sufficiently nice categories define a virtual equipment that model the
connectives of VETT. We highlight one example from the literature that
is highly general: Shulman's enriched indexed
categories~\cite{Shulman13}. Shulman's construction defines a virtual
double category of large and small $\mathcal V$-categories for any
pseudofunctor $\mathcal V : S^o \to \MonCat$ where $S$ is a category
with finite products. He gives examples that show that this subsumes
ordinary internal, enriched and indexed categories for suitable choices
of $\mathcal V$, as well as more general categories that can be thought
of as both indexed and enriched.
This is slightly weaker then what we require: to have
\emph{split} restrictions, we need that $\mathcal V$ be a
\emph{strict} functor, not merely a pseudo-functor.
This is analogous to the situation for dependent type theory, where
syntactic substitution is strictly associative, but semantic
substitution is typically given by pullback, which is only associative
up to unique isomorphism.
%
Shulman's construction carries over when the functor is strict but
some of their example instances would require a strictification
theorem.
\begin{construction}[Shulman \citet{Shulman13}]
  Given any functor $\mathcal V : S^o \to \SymMonCat$ such that $S$
  and $\mathcal V$ have sufficiently well-behaved (indexed)
  $\kappa$-products, then there is a virtual equipment $\mathcal
  V-\textrm{Cat}$ whose objects are locally $\kappa$-small $\mathcal
  V$-categories, small objects are $\kappa$-small $\mathcal
  V$-categories etc. This virtual equipment has objects satisfying all
  of the universal properties needed for a model of VETT.
\end{construction}

A final model that uses a CwF that is not $\Set$ would be given by
taking extensional dependent type theory as the CwF and interpreting
the category-theoretic constructions by their definitions inside type
theory.

\section{Related and Future Work}
\label{sec:discussion}

We now compare \vett{} with other calculi for formal category theory.  

C\'accamo and Winskel~\cite{caccamo_winskel_2001} develop a formal
language for defining categories, functors (of many variables) and
proving existence of natural equivalences between them. Their system
can encode profunctors as functors into $\Set$. Their natural
equivalence judgment does not have proof terms or equality between
equivalences and they do not support natural
transformations. Additionally, they only consider ordinary categories
as the intended model and do not develop a more general semantics.
Riehl and Verity~\cite{riehl_verity_2022} use a formal language of virtual
equipments to prove results valid for $\infty$-categories without
concrete manipulation of model categories. They formalize this language
as a theory in Makkai's framework of first-order logic with dependent
sorts (FOLDS).  While this previous work has the same models as \vett{},
we believe that the syntax we propose in this paper formalizes informal
arguments more directly, as shown in Section~\ref{sec:examples}.  This
is because FOLDS approach approach is entirely relational, whereas we
formalize concepts like restriction of a profunctor or composition of
natural transformations as functional operations (substitution).  In
particular, this means that our calculus requires only vertically
degenerate squares (elements/natural transformations) as a
``user-facing'' notion, with general squares occurring only in the
admissible substitution operations.


The coend calculus~\citet{loregian_2021} is an informal syntax for
manipulating profunctors involving ends and coends; an extension of
\vett{} to treat profunctors of many variables of different variances
may provide a formal treatment of it.

Myers~\citet{jazmyers-strings} provides a string diagram calculus for
double categories and pro-arrow equipments, generalizing string
diagrams for monoidal categories. These are an
alternative approach to type-theoretic calculi, with the string
diagrams typically making tensor products simpler to work with, while
a type-theoretic calculus like VETT makes the closed structure
$\homrallXYtoZ \alpha P Q$ simpler to work with by using bound
variables.

\emph{Cartesian bicategories} are similar to equipments but they
axiomatize the bicategory of profunctors rather than the full double
category of functors and profunctors \citet{CARBONI198711}.
Frey~\citet{frey} describes preliminary work on a proof system for
Cartesian bicateogires. Their profunctors are more general than in
\vett{} in as they may have 0, 1 or more covariant or contravariant
variables. But they do not have a term syntax for functors or natural
transformations.

Our work in this paper fits broadly into a line of work on
\emph{directed dependent type theories}, a type theory where the
identity type is interpreted as morphisms in a (possibly
$\infty$-)category.  In directed type theories based on a bisimplicial
model~\cite{riehlshulman17directed,buchholtzweinberger21fibered,weinberger22thesis,WeaverLicata20},
morphism types are defined using an interval object, like in cubical
type
theory~\cite{bch18,cchm18cubical,angiuli+18cartesian-csl,abcfhl17cartesian},
and universal properties like ``morphism induction'' are an internally
definable property of certain types.  Other type
theories~\cite{north18,ahrens+22semantics-2dtt} define morphism types
via an induction principle, corresponding to the lifting properties of
certain kinds of fibrations of categories.  While these previous works
can express some constructions on $\Cat$ that are not expressible in
\vett{}, because \vett{} is more restricted, \vett{} contrariwise has
more models, for instance categories enriched in non-cartesian
monoidal categories, so the theorems that are provable in \vett{}
apply in more settings.


Finally, some variations on double categories have been used to model
the structure of certain program
logics. GTT~\cite{double-cats-gradual-typing} is a logic for
\emph{vertically thin} pro-arrow equipments, where there is at most
one vertical arrow or 2-cell of any tyepe, so their calculus does not
include functor or transformation judgments. Another similar calculus
is System P~\cite{dunphyreddy} which is an internal language of
\emph{reflexive graph categories}, which are like double categories
without horizontal composition.


In future work, \vett{} could incorporate functor categories by
generalizing the unary type theory of functors to functors of many
variables, in which case ordinary $\lambda$ calculus can be used to
define functor categories as function types, and incorporate
multi-variable profunctors as in~\cite{frey}. This would require to
the models to have a monoidal structure. Ideas from coeffects and
enriched category theory may be useful for defining opposite
categories \cite{Shulman18,coeffect}.

\textbf{Acknowledgments.}  This material is based on research sponsored
by the National Science Foundation under agreement number CCF-1909517
and the United States Air Force Research Laboratory under agreement
number FA9550-21-0009 (Tristan Nguyen, program manager).
The authors would like to thank David Jaz Myers, Emily Riehl, Mike
Shulman, Dominic Verity for helpful feedback on this work.



\bibliographystyle{splncs04}
\bibliography{cats}

\begin{thebibliography}{10}
\providecommand{\url}[1]{\texttt{#1}}
\providecommand{\urlprefix}{URL }
\providecommand{\doi}[1]{https://doi.org/#1}

\bibitem{ahrens+22semantics-2dtt}
Ahrens, B., North, P., van~der Weide, N.: Semantics for two-dimensional type
  theory. In: ACM/IEEE Symposium on Logic in Computer Science (LICS) (2022)

\bibitem{altenkirch+16strict}
Altenkirch, T., Capriotti, P., Kraus, N.: Extending homotopy type theory with
  strict equality. In: EACSL Annual Conference on Computer Science Logic (CSL)
  (2016)

\bibitem{qit}
Altenkirch, T., Kaposi, A.: Type theory in type theory using quotient inductive
  types. In: Proceedings of the 43rd Annual ACM SIGPLAN-SIGACT Symposium on
  Principles of Programming Languages. p. 18–29. POPL '16 (2016).
  \doi{10.1145/2837614.2837638}

\bibitem{abcfhl17cartesian}
Angiuli, C., Brunerie, G., Coquand, T., {Hou (Favonia)}, K.B., Harper, R.,
  Licata, D.R.: Syntax and models of cartesian cubical type theory.
  Mathematical Structures in Computer Science  (2021)

\bibitem{angiuli+18cartesian-csl}
Angiuli, C., {Hou (Favonia)}, K.B., Harper, R.: Cartesian cubical computational
  type theory: Constructive reasoning with paths and equalities. In: Computer
  Science Logic (CSL) (2018)

\bibitem{awodeywarren09identity}
Awodey, S., Warren, M.: Homotopy theoretic models of identity types.
  Mathematical Proceedings of the Cambridge Philosophical Society  (2009)

\bibitem{bicategories}
B{\'e}nabou, J.: Introduction to bicategories. In: Reports of the Midwest
  Category Seminar. pp. 1--77. Springer Berlin Heidelberg, Berlin, Heidelberg
  (1967)

\bibitem{benabou2000distributors}
B{\'e}nabou, J.: Distributors at work. Lecture notes written by Thomas
  Streicher  \textbf{11} (2000)

\bibitem{bch18}
Bezem, M., Coquand, T., Huber, S.: The univalence axiom in cubical sets.
  Journal of Automated Reasoning  (June 2018). \doi{10.1007/s10817-018-9472-6}

\bibitem{topos-of-trees}
Birkedal, L., Møgelberg, R.E., Schwinghammer, J., Støvring, K.: {First steps
  in synthetic guarded domain theory: step-indexing in the topos of trees}.
  {Logical Methods in Computer Science}  \textbf{{Volume 8, Issue 4}} (Oct
  2012). \doi{10.2168/LMCS-8(4:1)2012}

\bibitem{coeffect}
Brunel, A., Gaboardi, M., Mazza, D., Zdancewic, S.: A core quantitative
  coeffect calculus. In: Proceedings of the 23rd European Symposium on
  Programming Languages and Systems - Volume 8410. p. 351–370 (2014).
  \doi{10.1007/978-3-642-54833-8_19}

\bibitem{buchholtzweinberger21fibered}
Buchholtz, U., Weinberger, J.: Synthetic fibered ($\infty$,1)-category theory,
  higher Structures, to appear. arXiv:2105.01724

\bibitem{caccamo_winskel_2001}
C{\'a}ccamo, M., Winskel, G.: A higher-order calculus for categories. In:
  Boulton, R.J., Jackson, P.B. (eds.) Theorem Proving in Higher Order Logics.
  pp. 136--153. Springer Berlin Heidelberg, Berlin, Heidelberg (2001)

\bibitem{CARBONI198711}
Carboni, A., Walters, R.: Cartesian bicategories i. Journal of Pure and Applied
  Algebra  \textbf{49}(1),  11--32 (1987).
  \doi{https://doi.org/10.1016/0022-4049(87)90121-6}

\bibitem{cartmell1986}
Cartmell, J.: Generalised algebraic theories and contextual categories. Annals
  of Pure and Applied Logic  \textbf{32},  209--243 (1986).
  \doi{https://doi.org/10.1016/0168-0072(86)90053-9}

\bibitem{cervesatopfenning02llf}
Cervesato, I., Pfenning, F.: A linear logical framework. Information and
  Computation  \textbf{179}(1),  19--75 (2002)

\bibitem{cchm18cubical}
Cohen, C., Coquand, T., Huber, S., M{\"o}rtberg, A.: Cubical type theory: A
  constructive interpretation of the univalence axiom. In: Uustalu, T. (ed.)
  21st International Conference on Types for Proofs and Programs (TYPES 2015).
  pp. 5:1--5:34 (2018). \doi{10.4230/LIPIcs.TYPES.2015.5}

\bibitem{Cruttwell2010}
Crutwell, G., Shulman, M.A.: A unified framework for generalized
  multicategories. Theory and Applications of Categories  \textbf{24},
  580--655 (2010)

\bibitem{CURIEN1986188}
Curien, P.L.: Categorical combinators. Information and Control  \textbf{69}(1),
   188--254 (1986). \doi{https://doi.org/10.1016/S0019-9958(86)80047-X}

\bibitem{diliberti-loregian}
Di~Liberti, I., Loregian, F.: On the unicity of formal category theories
  (2019). \doi{10.48550/ARXIV.1901.01594}

\bibitem{lslr}
Dreyer, D., Ahmed, A., Birkedal, L.: Logical step-indexed logical relations.
  In: 2009 24th Annual IEEE Symposium on Logic In Computer Science. pp. 71--80
  (2009). \doi{10.1109/LICS.2009.34}

\bibitem{dunphyreddy}
Dunphy, B.P., Reddy, U.S.: Parametric limits. In: 19th {IEEE} Symposium on
  Logic in Computer Science {(LICS} 2004), 14-17 July 2004, Turku, Finland,
  Proceedings. pp. 242--251 (2004). \doi{10.1109/LICS.2004.1319618}

\bibitem{dybjer-cwf}
Dybjer, P.: Internal type theory. In: Berardi, S., Coppo, M. (eds.) Types for
  Proofs and Programs. pp. 120--134. Springer Berlin Heidelberg (1996)

\bibitem{frey}
Frey, J.: A language for closed cartesian bicategories (2019), category Theory
  2019

\bibitem{girard}
Girard, J.Y.: Linear logic. Theoretical Computer Science  \textbf{50}(1),
  1--101 (1987). \doi{https://doi.org/10.1016/0304-3975(87)90045-4}

\bibitem{grandis-pare99}
Grandis, M., Pare, R.: Limits in double categories. Cahiers de Topologie et
  G\'eom\'etrie Diff\'erentielle Cat\'egoriques  \textbf{40}(3),  162--220
  (1999)

\bibitem{hofmann98groupoid}
Hofmann, M., Streicher, T.: The groupoid interpretation of type theory. In:
  Twenty-five years of constructive type theory. Oxford University Press (1998)

\bibitem{isaev21indexed}
Isaev, V.: Indexed type theories. Mathematical Structures in Computer Science
  \textbf{31}(1),  3–63 (2021). \doi{10.1017/S0960129520000092}

\bibitem{iris}
Jung, R., Swasey, D., Sieczkowski, F., Svendsen, K., Turon, A., Birkedal, L.,
  Dreyer, D.: Iris: Monoids and invariants as an orthogonal basis for
  concurrent reasoning. In: Proceedings of the 42nd Annual ACM SIGPLAN-SIGACT
  Symposium on Principles of Programming Languages. p. 637–650. POPL '15,
  Association for Computing Machinery (2015). \doi{10.1145/2676726.2676980}

\bibitem{krishnaswami+15linear}
Krishnaswami, N.R., Pradic, P., Benton, N.: Integrating dependent and linear
  types. In: ACM Symposium on Principles of Programming Languages (2015)

\bibitem{LambekScott}
Lambek, J., Scott, P.: Introduction to Higher-Order Categorical Logic.
  Cambridge University Press (1988)

\bibitem{Lawvere69}
Lawvere, F.W.: Adjointness in foundations. Dialectica  \textbf{23} (1969)

\bibitem{LEINSTER2002391}
Leinster, T.: Generalized enrichment of categories. Journal of Pure and Applied
  Algebra  \textbf{168}(2),  391--406 (2002).
  \doi{https://doi.org/10.1016/S0022-4049(01)00105-0}, category Theory 1999:
  selected papers, conference held in Coimbra in honour of the 90th birthday of
  Saunders Mac Lane

\bibitem{loregian_2021}
Loregian, F.: (Co)end Calculus. London Mathematical Society Lecture Note
  Series, Cambridge University Press (2021). \doi{10.1017/9781108778657}

\bibitem{moggi}
Moggi, E.: Notions of computation and monads. Information and Computation
  \textbf{93}(1),  55--92 (1991).
  \doi{https://doi.org/10.1016/0890-5401(91)90052-4}, selections from 1989 IEEE
  Symposium on Logic in Computer Science

\bibitem{jazmyers-strings}
Myers, D.J.: String diagrams for double categories and equipments (2016).
  \doi{10.48550/ARXIV.1612.02762}

\bibitem{double-cats-gradual-typing}
New, M.S., Licata, D.R.: {Call-by-Name Gradual Type Theory}. In: Kirchner, H.
  (ed.) 3rd International Conference on Formal Structures for Computation and
  Deduction (FSCD 2018). Leibniz International Proceedings in Informatics
  (LIPIcs), vol.~108, pp. 24:1--24:17. Schloss Dagstuhl--Leibniz-Zentrum fuer
  Informatik, Dagstuhl, Germany (2018). \doi{10.4230/LIPIcs.FSCD.2018.24}

\bibitem{extended-version}
New, M.S., Licata, D.R.: A formal logic for formal category theory (extended
  version) (2022). \doi{10.48550/ARXIV.2210.08663},
  \url{https://arxiv.org/abs/2210.08663}

\bibitem{north18}
North, P.R.: Towards a directed homotopy type theory. In: Mathematical
  Foundations of Programming Semantics (MFPS) (2019)

\bibitem{PALMGREN2019102715}
Palmgren, E.: Categories with families and first-order logic with dependent
  sorts. Annals of Pure and Applied Logic  \textbf{170}(12),  102715 (2019).
  \doi{https://doi.org/10.1016/j.apal.2019.102715},
  \url{https://www.sciencedirect.com/science/article/pii/S0168007219300727}

\bibitem{abadi-plotkin}
Plotkin, G., Abadi, M.: A logic for parametric polymorphism. In: Bezem, M.,
  Groote, J.F. (eds.) Typed Lambda Calculi and Applications. pp. 361--375.
  Springer Berlin Heidelberg, Berlin, Heidelberg (1993)

\bibitem{polakow-pfenning}
Polakow, J., Pfenning, F.: Natural deduction for intuitionistic
  non-communicative linear logic. In: Girard, J. (ed.) Typed Lambda Calculi and
  Applications, 4th International Conference, TLCA'99, L'Aquila, Italy, April
  7-9, 1999, Proceedings. Lecture Notes in Computer Science, vol.~1581, pp.
  295--309. Springer (1999). \doi{10.1007/3-540-48959-2\_21}

\bibitem{riehlshulman17directed}
Riehl, E., Shulman, M.: A type theory for synthetic $\infty$-categories. Higher
  Structures  \textbf{1}(1) (2018)

\bibitem{riehl_verity_2022}
Riehl, E., Verity, D.: Elements of $\infty$-Category Theory. Cambridge Studies
  in Advanced Mathematics, Cambridge University Press (2022).
  \doi{10.1017/9781108936880}

\bibitem{reflexive-graphs}
Robinson, E., Rosolini, G.: Reflexive graphs and parametric polymorphism. In:
  Proceedings Ninth Annual IEEE Symposium on Logic in Computer Science. pp.
  364--371 (1994). \doi{10.1109/LICS.1994.316053}

\bibitem{seely_1984}
Seely, R.A.G.: Locally cartesian closed categories and type theory.
  Mathematical Proceedings of the Cambridge Philosophical Society
  \textbf{95}(1),  33–48 (1984). \doi{10.1017/S0305004100061284}

\bibitem{Shulman08}
Shulman, M.: Framed bicategories and monoidal fibrations. Theory and
  Applications of Categories  \textbf{20},  650--738 (2008),
  \url{http://www.tac.mta.ca/tac/volumes/20/18/20-18abs.html}

\bibitem{Shulman13}
Shulman, M.: Enriched indexed categories. Theory and Applications of Categories
   \textbf{28},  616--695 (2013),
  \url{http://www.tac.mta.ca/tac/volumes/28/21/28-21abs.html}

\bibitem{Shulman18}
Shulman, M.: Contravariance through enrichment. Theory and Applications of
  Categories  \textbf{33},  95--130 (2018),
  \url{http://tac.mta.ca/tac/volumes/33/5/33-05abs.html}

\bibitem{STREET1978350}
Street, R., Walters, R.: Yoneda structures on 2-categories. Journal of Algebra
  \textbf{50}(2),  350--379 (1978).
  \doi{https://doi.org/10.1016/0021-8693(78)90160-6},
  \url{https://www.sciencedirect.com/science/article/pii/0021869378901606}

\bibitem{voevodsky06homotopy}
Voevodsky, V.: A very short note on homotopy $\lambda$-calculus (September
  2006), unpublished.

\bibitem{voevodsky13hts}
Voevodsky, V.: A type system with two kinds of identity types (2013), talk at
  Andre Joyal's 70th birthday conference (IAS)

\bibitem{vakar15linear}
Vákár, M.: A categorical semantics for linear logical frameworks. In:
  Foundations of Software Science and Computation Structures (FoSSaCS) (2015)

\bibitem{wand197913}
Wand, M.: Fixed-point constructions in order-enriched categories. Theoretical
  Computer Science  \textbf{8}(1),  13--30 (1979).
  \doi{https://doi.org/10.1016/0304-3975(79)90053-7}

\bibitem{WeaverLicata20}
Weaver, M.Z., Licata, D.R.: A constructive model of directed univalence in
  bicubical sets. In: Proceedings of the 35th Annual ACM/IEEE Symposium on
  Logic in Computer Science. pp. 915--928. LICS '20, Association for Computing
  Machinery, New York, NY, USA (2020). \doi{10.1145/3373718.3394794}

\bibitem{weinberger22thesis}
Weinberger, J.: A Synthetic Perspective on ($\infty$,1)-Category Theory:
  Fibrational and Semantic Aspects. Ph.D. thesis, TU Darmstadt (2022),
  arXiv:2202.13132

\bibitem{proarrow-equipments}
Wood, R.J.: Abstract pro arrows {I}. Cahiers de Topologie et G\'eom\'etrie
  Diff\'erentielle Cat\'egoriques  \textbf{23}(3),  279--290 (1982)

\end{thebibliography}

\ifextended
\appendix

\section{Details of \vett{} Syntax and Syntactic Metatheory}
\label{sec:appendix:syntax}

\subsection{Contexts and Substitutions}

In Figure~\ref{fig:ctx} we include the formation rules and definitions
of the three kinds of contexts that are used in \vett{}.
In Figure~\ref{fig:substitutions} we give definitions for
well-typedness of the corresponding three kinds of substitutions.

\begin{figure}
  \begin{mathpar}
    \inferrule*[right=TyCtxForm]
    {}
    {\Gamma \isadtctx}

    \inferrule*[right=MtTyCtx]{}{\cdot \isadtctx}

    \inferrule*[right=TyCtxExt]
    {\Gamma \isadtctx \and \Gamma \vdash A \isaTy}
    {\Gamma , X : A \isadtctx}

    \inferrule*[right=BoundaryForm]
    {\Gamma \isadtctx}
    {\Gamma \vdash \Xi \boundary}

    \inferrule*[right=BoundarySingle]
    {\Gamma \vdash \cat C \isaCat}
    {\Gamma \vdash \alpha: \cat C \boundary}

    \inferrule*[right=BoundaryDbl]
    {\Gamma \vdash \cat C \isaCat \and \Gamma\vdash\cat D \isaCat}
    {\Gamma \vdash \alpha: \cat C; \beta: \cat D \boundary}

    \inferrule*[right=TransCtxForm]
    {\Gamma \isadtctx}
    {\Gamma \vdash \Phi \isavectx}

    \inferrule*[right=TransCtxMt]
    {\Gamma \vdash \cat C \isaCat}
    {\Gamma \vdash \alpha: \cat C \isavectx}

    \inferrule*[right=TransCtxExt]
    {\Gamma \vdash \Phi \isavectx\and
    \Gamma \pipe d^+\Phi; \beta:\cat D \vdash R \isaSet}
    {\Gamma \vdash \Phi, x:R, \beta:\cat D \isavectx}
  \end{mathpar}
  \caption{Contexts}
  \label{fig:ctx}
\end{figure}

\begin{figure}
  \begin{mathpar}
    \begin{array}{rccl}
      \textrm{Term Substitution} & \gamma, \delta & ::= & \cdot \pipe \gamma, M/X\\
      \textrm{Object Substitution} & \xi, \zeta & ::= & a/\alpha \pipe a/\alpha;b/\beta\\
      \textrm{Transformation Substitution} & \phi, \psi & ::= & a/\alpha \pipe \phi,s/x, \psi      
    \end{array}

    \inferrule*[right=TermSubstFormation]
    {\Delta \isadtctx \and \Gamma \isadtctx}
    {\Delta \vdash \gamma :: \Gamma}
    
    \inferrule*[right=TermSubstMt]
    {}
    {\Delta \vdash \cdot :: \cdot}

    \inferrule*[right=TermSubstExt]
    {\Delta \vdash \gamma :: \Gamma \and \Delta \vdash M : A[\gamma]}
    {\Delta \vdash \gamma,M/X :: \Gamma,X:A}

    \inferrule*[right=BoundarySubstFormation]
    {\Gamma \vdash Z \boundary \and \Gamma \vdash \Xi \boundary}
    {\Gamma\pipe Z \vdash \xi :: \Xi}

    \inferrule*[right=BoundarySubstSingle]
    {\Gamma\pipe \alpha:\cat C \vdash b : \cat D}
    {\Gamma \pipe \alpha:\cat C \vdash b/\beta :: \beta:\cat D}

    \inferrule*[right=BoundarySubstDbl]
    {\Gamma\pipe d^-\Xi \vdash a : \cat C\and
     \Gamma\pipe d^+\Xi \vdash b : \cat D}
    {\Gamma \pipe \Xi \vdash a/\alpha;b/\beta :: (\alpha : \cat C; \beta:\cat D)}

    \inferrule*[right=ElementSubstFormation]
    {\Gamma \vdash \Psi \isavectx \and \Gamma \vdash \Phi \isavectx}
    {\Gamma\pipe \Psi \vdash \phi :: \Phi}

    \inferrule*[right=ElementSubstMt]
    {\Gamma\pipe \beta:\cat C \vdash a : \cat C}
    {\Gamma \pipe \beta:\cat C \vdash a/\alpha : (\alpha : \cat C)}

    \inferrule*[right=ElementSubstExt]
    {\Gamma\pipe\Psi \vdash \phi :: \Phi \and
      d^+\Psi = d^-\Psi'
     \Gamma\pipe\Psi' \vdash t : R[d^+\phi;b/\beta]\and
     \Gamma\pipe d^+\Psi'\vdash b : \cat D}
    {\Gamma\pipe\Psi \vdash \phi,t/x,b/\beta :: (\Phi,x:R,\beta:\cat D)}
  \end{mathpar}
  \caption{Substitution}
  \label{fig:substitutions}
\end{figure}

These definitions involve several operations $d^\pm\Phi$, $\underline
\Phi$ and $\Phi\jnctx \Psi$ on contexts (and their functorial lift to
substitutions) that we now define.

\begin{definition}
  We define operations $d^\pm$ that project out the covariant and
  contravariant boundary of a set context.
  This can be typed with the admissible rule
  \begin{mathpar}
    \inferrule*[right=(*)]
               {\Gamma \vdash \Phi \isavectx}
               {\Gamma \vdash d^\pm\Phi \boundary}
  \end{mathpar}
  This is defined as
  \begin{align*}
    d^{\pm}(\alpha:\cat C) &= \alpha : \cat C\\
    d^{-}(\Phi,x:R,\Psi) &= d^-\Phi\\
    d^{+}(\Phi,x:R,\Psi) &= d^-\Psi
  \end{align*}
  This operation extends to the substitutions with admissible rule
  \begin{mathpar}
    \inferrule*[right=(*)]
    {\Gamma\pipe\Psi \vdash \phi :: \Phi}
    {\Gamma\pipe d^\pm \Psi \vdash d^\pm\phi :: d^\pm \Phi}
  \end{mathpar}
  defined as
  \begin{align*}
    d^{\pm}(a/\alpha) &= a/\alpha\\
    d^{-}(\phi,t/x,\psi) &= d^-\phi\\
    d^{+}(\phi,t/x,\psi) &= d^+\psi\\
  \end{align*}
\end{definition}

  Note that $d^\pm\Phi$ will always be a set context with a single
  variable $\alpha : \cat C$---we exploit the fact that we have these
  singleton set contexts to avoid introducing a separate syntactic class
  of category contexts $\alpha : \cat C$ and substitutions between them.

\begin{definition}
  We define the operation of restricting a set context to both sides
  of its boundary with admissible typing
  \begin{mathpar}
    \inferrule*[right=(*)]
    {\Gamma \vdash \Phi \isavectx}
    {\Gamma \vdash \underline\Phi \boundary}
  \end{mathpar}
  and definition
  \begin{align*}
    \underline{\alpha:\cat C} &= \alpha:\cat C\\
    \underline{\Phi,x:R,\Psi} &= d^-\Phi; d^+{\Psi}
  \end{align*}
  The extension to substitutions has admissible typing
  \begin{mathpar}
    \inferrule*[right=(*)]
    {\Gamma \pipe \Psi \vdash \phi :: \Phi}
    {\Gamma\pipe\underline \Psi \vdash \underline \phi : \underline \Phi}
  \end{mathpar}
  and definition
  \begin{align*}
    \underline{b/\beta} &= b/\beta\\
    \underline{\phi} &= d^-\phi; d^+{\phi} \text{ otherwise }
  \end{align*}
\end{definition}

Finally, we define the operation of ``horizontal composition'' of set
contexts $\Phi \jnctx \Psi$ and its functorial lift $\phi \jnctx
\psi$.
\begin{definition}
  We define horizontal composition of transformation contexts with the
  admissible typing rule
  \[\inferrule*[right=(*)]
  {\Gamma\vdash \Phi\isavectx\and \Gamma\vdash \Psi\isavectx\and d^+\Phi=d^-\Psi}
  {\Gamma\vdash \Phi \jnctx \Psi \isavectx }\]
  as follows
  \begin{align*}
    \Phi \jnctx \alpha:\cat C &= \Phi\\
    \Phi \jnctx (\Psi,x:R,\alpha:\cat C) &= (\Phi\jnctx \Psi), x:R, \alpha:\cat C
  \end{align*}

  And we extend this to an operation on substitutions with the
  admissible rule
  \[\inferrule*[right=(*)]
  {\Gamma\pipe\Psi \vdash \phi : \Phi \and
   \Gamma\pipe\Psi' \vdash \phi' : \Phi'\and
   d^+\phi = d^-\phi'
  }
  {\Gamma\pipe\Psi \jnctx \Psi' \vdash \phi \jnctx \phi' :: \Phi \jnctx \Phi' }
  \]
  Defined as follows
  \begin{align*}
    \phi \jnctx a/\alpha &= \phi\\
    \phi\jnctx (\psi, s/x,a/\alpha) &= (\phi\jnctx\psi), s/x, a/\alpha
  \end{align*}  
\end{definition}

\begin{lemma}[Horizontal Category of Contexts/Substitutions]
  Horizontal composition of contexts is associative (when defined)
  \[ (\Phi \jnctx \Psi) \jnctx \Sigma = \Phi \jnctx (\Psi \jnctx \Sigma) \]
  and unital with identity for $\cat C$ given by the single category
  variable context $\alpha:\cat C$:
  \[ \alpha:\cat C \jnctx \Phi = \Phi = \Phi \jnctx \beta:\cat D\]
  when $d^-\Phi = \alpha:\cat C$ and $d^+\Phi = \beta:\cat D$.

  These properties extend to the horizontal composition of element substitutions:
  \[ \phi \jnctx (\psi \jnctx \sigma) = (\phi \jnctx \psi) \jnctx \sigma \]
  where the identity is the single variable substitution:
  \[ a/\alpha \jnctx \phi = \phi = \phi \jnctx b/\beta \]
  when $d^-\phi = a/\alpha$ and $d^+\phi = b/\beta$.
\end{lemma}

Next, we define the actions of substitutions on terms. We elide the
obvious action of term substitutions $\gamma$ and include only the
more unusual substructural substitutions.
\begin{definition}[Substitution Actions]
  For any $\Gamma \pipe \alpha: \cat C \vdash a : \cat D$ and $\Gamma
  \pipe \beta:\cat D \vdash b : \cat E$, we define
  $\Gamma \pipe \alpha : \cat C \vdash b[a/\beta] : \cat E$ by recursion on $b$:

  \begin{align*}
    \beta[a/\beta] &= a\\
    (M\,b)[a/\beta] &= M\,(b[a/\beta])\\
    (b_1,b_2)[a/\beta] &= (b_1[a/\beta],b_2[a/\beta])\\
    (\pi_i b)[a/\beta] &= \pi_i b[a/\beta]\\
    () [a/\beta] &= () \\
    (\pi_\pm b)[a/\beta] &= \pi_\pm b[a/\beta]\\
    (b_-,b_+,s)[a/\beta] &= (b_-[a/\beta],b_+[a/\beta], s[a/\beta])\\
    (\lambda \alpha.R)[a/\beta] &=\lambda \alpha.R[a/\beta] 
  \end{align*}

  Simultaneously,
  for
  $\Gamma \pipe \Psi \vdash \phi : \Phi$
  and 
  $\Gamma \pipe \Phi \vdash s : R$ we define\\
  $\Gamma \pipe \Psi \vdash s[\phi] : R [ \underline \phi ]$ by recursion on $s$:
  \begin{align*}
    x[a/\alpha,t/x,b/\beta] &= t\\
    \pendappXtoY M b[a/\alpha] &= \pendappXtoY M {b[a/\alpha]}\\
    {\punitelimkontZatABC {\alpha. t}{b_1} s {b_2}}[\phi] &=
    {\punitelimkontZatABC {\alpha. t}{b_1[d^-\phi]} {s[\phi]} {b_2[d^+\phi]}}\\
    (\punitrefl b)[a/\alpha] &= \punitrefl {b[a/\alpha]}\\
    {\tensorelimWkontZ s {x,\beta,y.r}[\phi_l\jnctx\phi_m\jnctx\phi_r]} &=
    \tensorelimWkontZ {s[\phi_m]} {x,\beta,y. r[\phi_l,x/x,\beta/\beta \jnctx \beta/\beta,y/y,\phi_r]}\\
    {\tensorintroatXwithYandZ b s t}[\phi_s \jnctx \phi_t] &= \tensorintroatXwithYandZ {b[d^+\phi_s]} {s[\phi_s]} {t[\phi_t]}\\
    {(\homrappXtoYatZ s t a)}[\phi_f\jnctx \phi_a] &=
    {\homrappXtoYatZ {s[\phi_f]} {t[\phi_a]} {a[d^+\phi_f]}}\\
    {(\homrlambdaXatYdotZ {x}{\alpha} s)}[\phi] &=
    {\homrlambdaXatYdotZ {x}{\alpha} {s[\phi,x/x,\alpha/\alpha]}}\\
    {(\homlappXtoYatZ t s a)}[\phi_a\jnctx \phi_f] &=
    {\homlappXtoYatZ {t[\phi_a]} {s[\phi_f]} {a[d^-\phi_f]}}\\
    {(\homllambdaXatYdotZ {x}{\alpha} s)}[\phi] &=
    {\homllambdaXatYdotZ {x}{\alpha} {s[\alpha/\alpha,x/x,\phi]}}\\
    (\pi_i s)[\phi] &= \pi_i s[\phi]\\
    (s_1,s_2)[\phi] &= (s_1[\phi], s_2[\phi])\\
    ()[\phi] &= ()
  \end{align*}
  Several rules assume the substitution is in a particular form, such
  as the tensor elimination which expects an input context
  $\phi_s\jnctx\phi_t$. The fact that the context can be uniquely
  decomposed in a well-typed way follows from an inversion principle (Lemma~\ref{lem:inversion})
  for well-typed substitutions.

  And finally, for $\Gamma \pipe \Xi' \vdash \xi : \Xi$ and $\Gamma \pipe \Xi \vdash P \isaSet$,
  we define $\Gamma \pipe \Xi' \vdash P[\xi] \isaSet$ by recursion on $P$:
  \begin{align*}
    (M\,a\,b)[\xi] &= M\,(a[d^-\xi])\,(b[d^+\xi])\\
    (\punitinXfromYtoZ {\cat C} a b)[\xi] &= \punitinXfromYtoZ {\cat C} {a[d^- \xi]}{b[d^+ \xi]}\\
    (\tensorexistsXwithYandZ \beta P Q)[\xi] &= \tensorexistsXwithYandZ \beta {P[d^-\xi;\beta/\beta]}{P[\beta/\beta;d^+\xi]}\\
    (\homrallXYtoZ \alpha R P)[\xi] &= \homrallXYtoZ \alpha {R[d^+\xi;\alpha/\alpha]} {P[d^-\xi;\alpha/\alpha]}\\
    (\homlallXYtoZ \alpha P Q)[\xi] &= \homlallXYtoZ \alpha {P[\alpha/\alpha;d^-\xi]} {Q[\alpha/\alpha;d^+\xi]}\\
    1[\xi] &= 1\\
    (P_1 \times P_2)[\xi] &= P_1[\xi] \times P_2[\xi]
  \end{align*}
\end{definition}

\begin{lemma}[Inversion]
  \label{lem:inversion}
  \begin{enumerate}
  \item If $\Phi \vdash \psi :: (\alpha:\cat C)$ then $\Phi =
    \beta:\cat D$ for some $\cat D$ and $\psi = a/\alpha$ where
    $\beta:\cat D \vdash a : \cat C$.
  \item If $\Phi \vdash \psi :: \Psi_1 \jnctx \Psi_2$, then there
    exists unique $\Phi_1,\Phi_2,\psi_1,\psi_2$ such that $\Phi =
    \Phi_1\jnctx\Phi_2$ and $\Phi_1\vdash\psi_1 :: \Psi_1$ and
    $\Phi_2\vdash \psi_2 : \Psi_2$ and $\psi = \psi_1\jnctx\psi_2$.
  \end{enumerate}
\end{lemma}





\subsection{Equational Theory}

Next we present the $\beta\eta$ rules that generate the equational
theory of the terms. In keeping with the extensional style of the type
theory, we do not present explicit transitivity, congruence, or
transport rules, but rather consider these as inherent to the notion
of equality. This can be formalized by modeling the terms of our type
theory as a quotient inductive inductive type
\cite{qit}.  We elide the types on the $\beta$ rules, as they can be
inferred from the shape of the term, but include them for clarity on the
$\eta$ rules.  

\begin{figure}
  \begin{mathpar}
    \inferrule*[right=SmallCat$\beta$]
    {~}
    {\unquoth {\quoth {\cat C}}  = \cat C}

    \inferrule*[right=SmallCat$\eta$]
    {\Gamma \vdash M : \smallCats}
    {\Gamma \vdash \quoth {\unquoth {M}}  = M : \smallCats}

    \inferrule*[right=Cat$\beta$]
    {~}
    {\unquoth {\quoth {\cat C}}  = \cat C}
    
    \inferrule*[right=Cat$\eta$]
    {\Gamma \vdash M : \Cats}
    {\Gamma \vdash \quoth {\unquoth {M}}  = M : \Cats}

    \\
    \inferrule*[right=Fctor$\beta$]
    {~}
    {(\lambda \alpha. b)\,a = b[a/\alpha]}

    \inferrule*[right=Fctor$\eta$]
    {\Gamma \vdash M : \varr {\cat C}{\cat D}}
    {\Gamma \vdash M = \lambda \alpha. M\,\alpha : \varr {\cat C}{\cat D}}

    \inferrule*[right=Prof$\beta$]
    {~}
    {(\lambda \alpha \beta. P)\,a\,b = P[a/\alpha;b/\beta]}

    \inferrule*[right=Prof$\eta$]
    {\Gamma \vdash M : \harr {\cat C}{\cat D}}
    {\Gamma \vdash M = \lambda \alpha\,\beta. M\,\alpha\,\beta : \harr {\cat C}{\cat D}}

    \inferrule*[right=NatElt$\beta$]
    {~}
    {(\pendlambdaXdotY \alpha s) a = s[a/\alpha]}

    \inferrule*[right=NatElt$\eta$]
    {\Gamma \vdash M : \pendallXdotY \alpha P}
    {\Gamma \vdash M = \pendlambdaXdotY \alpha {M\,a} : \pendallXdotY\alpha P}

    \\\\\\\\


    \inferrule*[right=NegPresheaf$\beta$]
    {~}
    {\negPresheafApp a {(\lambda \alpha. R)} = R[a/\alpha]}

    \inferrule*[right=NegPresheaf$\eta$]
    {\Gamma \pipe \beta:\cat D \vdash p : \negPresheaf {\cat C}}
    {p = \lambda \alpha. \negPresheafApp \alpha p}
    
    \inferrule*[right=NegPresheaf$\beta$]
    {~}
    {\posPresheafApp {(\lambda \beta. R)} b = R[b/\beta]}

    \inferrule*[right=PosPresheaf$\eta$]
    {\Gamma \pipe \alpha:\cat C \vdash p : \posPresheaf {\cat D}}
    {p = \lambda \beta. \posPresheafApp p \beta}
    
    \inferrule*[right=Graph$\beta-$]
    {~}
    {\pi_- (a_-,a_+,s) = a_-}

    \inferrule*[right=Graph$\beta+$]
    {~}
    {\pi_+ (a_-,a_+,s) = a_+}
    
    \inferrule*[right=Graph$\beta e$]
    {~}
    {\pi_e (a_-,a_+,s) = s}
    
    \inferrule*[right=Graph$\eta$]
    {\Gamma \pipe \alpha:\cat C \vdash b : \graphProf{\beta_-}{\beta_+} P}
    {\Gamma \pipe \alpha:\cat C \vdash b = (\pi_-b, \pi_+b, \pi_e b) : \graphProf{\beta_-}{\beta_+} P}
    
  \inferrule*[right=${1}\eta$]
  {\Gamma \pipe \alpha:\cat C \vdash a : 1}
  {\Gamma \pipe \alpha:\cat C \vdash a = () : 1}\\

  \inferrule*[right=${\times}\beta$]
  {~}
  {\pi_i(a_1,a_2) = a_i}

  \inferrule*[right=${\times}\eta$]
  {\Gamma \pipe \alpha:\cat C \vdash a : \cat D_1 \times \cat D_2}
  {\Gamma \pipe \alpha:\cat C \vdash a = (\pi_1a, \pi_2a) : \cat D_1 \times \cat D_2}
  \end{mathpar}
  \caption{$\beta\eta$ Equality for type and object connectives}
\end{figure}

\begin{figure}
  \begin{mathpar}
  \inferrule*[right=CovHom$\beta$]
    {~}
    {\homrappXtoYatZ {(\homrlambdaXatYdotZ x \alpha s)} t a = s[t/x,a/\alpha]}
    \inferrule*[right=CovHom$\eta$]
               {\Gamma\pipe\Phi \vdash s : \homrallXYtoZ \alpha R P}
               {\Gamma\pipe\Phi \vdash s = \homrlambdaXatYdotZ {x}{\alpha} \homrappXtoYatZ s x \alpha}

    \inferrule*[right=ConHom$\beta$]
                 {~}
                 {\homlappXtoYatZ {(\homllambdaXatYdotZ x \alpha t)} s a = t[a/\alpha,s/x]}
      \inferrule*[right=ConHom$\eta$]
      {\Gamma \pipe \Phi \vdash t : \homlallXYtoZ \alpha R P}
      {\Gamma \pipe \Phi \vdash t = \homllambdaXatYdotZ x \alpha \homlappXtoYatZ t x \alpha : \homlallXYtoZ \alpha R P}

    \inferrule*[right=Unit$\beta$]
    {~}
    {(\punitelimkontZatABC {\alpha. t} a {\punitrefl a} a) = t[a/\alpha]}
    
    \inferrule*[right=Unit$\eta$]
    {\Gamma\pipe \alpha_1:\cat C,z:\punitinXfromYtoZ {\cat C}{\alpha_1}{\alpha_2},\alpha_2:\cat C \vdash s : R}
    {\Gamma\pipe\alpha_1:\cat C,z:\punitinXfromYtoZ {\cat C}{\alpha_1}{\alpha_2},\alpha_2:\cat C\vdash s = \punitelimkontZatABC {\alpha. s[\alpha/\alpha_1;\punitrefl{\alpha}/z,\alpha/\alpha_2]} {\alpha_1} z {\alpha_2}: R}
      
    \inferrule*[right=Tensor$\beta$]
    {~}
    {\tensorelimWkontZ {\tensorintroatXwithYandZ b s t} {x,\beta,y. r} = t[s/x;b/\beta;t/y]}
    
    \inferrule*[right=Tensor$\eta$]
    {\Gamma\pipe\Phi,z:\tensorexistsXwithYandZ \beta P Q, \Psi \vdash s : R}
    {\Gamma\pipe\Phi,z:\tensorexistsXwithYandZ \beta P Q, \Psi \vdash s = \tensorelimWkontZ s {x,\beta,y. s[\tensorintroatXwithYandZ \beta x y/z]}: \tensorexistsXwithYandZ \beta P Q}

  \inferrule*[right=$1\eta$]
  {\Phi \vdash s : 1}
  {\Phi \vdash s = () : 1}\\

  \inferrule*[right=${\times}\beta$]
  {~}
  {\pi_i(s_1,s_2) = s_i}

  \inferrule*[right=${\times}\eta$]
  {\Phi \vdash s : P \times Q}
  {\Phi \vdash s = (\pi_1 s, \pi_2 s) : P \times Q}
  \end{mathpar}
  \caption{$\beta\eta$ Equality for set connectives}
  \label{fig:betaeta-sets}
\end{figure}

\subsection{Generalized Unit Elimination}

The unit elimination rule presented in Section~\ref{sec:syntax} is
more restrictive than universal property of a unit in a virtual double
category that we use in the semantics. So in order for our calculus to
be complete for virtual equipments with units, we need to show that
the more general unit elimination principle is admissible and
satisfies the correct $\beta\eta$ rules.

The more general rules are as follows
\begin{mathpar}
  \inferrule*[Right=UnitElim]
  {\Phi[\alpha/\alpha_-]\jnctx \Psi[\alpha/\alpha_+] \vdash t : P[\alpha/\alpha_-;\alpha/\alpha_+]}
  {\Phi,x:\punitinXfromYtoZ {}{\alpha_-} {\alpha_+}, \Psi \vdash \textrm{ind}_{\to}^{\Phi;\Psi}(\alpha.t; x) : P}

  (\textrm{ind}_{\to}^{\Phi;\Psi}(\alpha.t; x))[\phi,\punitrefl \alpha,\psi] = t[\phi\jnctx\psi]

  \inferrule
  {\Phi,x:\punitinXfromYtoZ {}{\alpha_-} {\alpha_+}, \Psi \vdash s : P}
  {\Phi,x:\punitinXfromYtoZ {}{\alpha_-} {\alpha_+}, \Psi \vdash s = \textrm{ind}_{\to}^{\Phi;\Psi}(\alpha. s[\alpha/\alpha_-,\punitrefl \alpha/x,\alpha/\alpha_+]; x) : P}
\end{mathpar}

The rule is more general because it allows the elimination of an input
of the unit type with non-trivial contexts $\Phi, \Psi$ surrounding
it, whereas the rule presented earlier would only allow this
elimination if $x$ were the only variable. We did not include this
more general rule as a basic inference rule because 
it requires an additional explicit substitution for the context
$\Phi,x:\punitinXfromYtoZ {}{\alpha_-} {\alpha_+}, \Psi$,
which would require making the substitutions part of the basic syntax.
In the presence of hom types, we can prove this more general
elimination is admissible, because the judgment
\[ \Phi,x:\punitinXfromYtoZ {}{\alpha_-} {\alpha_+}, \Psi \vdash s : P \]
is in natural bijection with the judgment
\[ \alpha_-,x:\punitinXfromYtoZ {}{\alpha_-} {\alpha_+}, \alpha_+ \vdash s : \Phi \triangleright P \triangleleft \Psi \]
where $\Psi \triangleright P \triangleleft \Phi$ (note the reversal of
order) is a type constructed by recursion on $\Phi$ and $\Psi$ using
uses the hom types.  The function applications for the hom types then
provide a more lightweight way to incorporate the explicit substitution
into the definition of the type theory.  

\begin{definition}[Generalized Unit Elimination]
  We define $\textrm{ind}_{\to}^{\Phi;\Psi}(\alpha.t; x)$ by induction
  on $\Phi/\Psi$.
  \begin{align*}
    \textrm{ind}_{\to}^{\alpha_-;\alpha_+}(\alpha.t; x) &= \punitelimkontZatABC{\alpha.t}{\alpha_-}{x}{\alpha_+}\\
    \textrm{ind}_{\to}^{\alpha_-;\Psi,y,\beta}(\alpha.t; x) &=
    \homrappXtoYatZ
        {(\textrm{ind}_{\to}^{\alpha_-;\Psi}(\alpha.\homrlambdaXatYdotZ y \beta t; x))} y \beta \\
    \textrm{ind}_{\to}^{\beta, y, \Phi;\Psi}(\alpha.t; x) &=
    \homlappXtoYatZ
        {(\textrm{ind}_{\to}^{\alpha_-;\Psi}(\alpha.\homllambdaXatYdotZ y \beta t; x))} y \beta \\
  \end{align*}
\end{definition}

\begin{lemma}[Generalized Unit Elim $\beta\eta$]
  The admissible generalized unit elimination satisfies the described $\beta\eta$ equations.
\end{lemma}
\begin{proof}
  By induction on $\Phi/\Psi$
  First, $\beta$
  \begin{itemize}
  \item If $\Phi = \Psi = \alpha$
    \begin{align*}
    (\punitelimkontZatABC{\alpha.t}{\alpha_-}{x}{\alpha_+})[\alpha,\punitrefl \alpha, \alpha] &=
    \punitelimkontZatABC{\alpha.t}{\alpha}{\punitrefl}{\alpha}\\
    &= t \tag{Unit$\beta$}
  \end{align*}
  \item If $\Phi = \alpha$ and $\Psi = \Psi, y, \beta$
    \begin{align*}
    (\homrappXtoYatZ
      {(\textrm{ind}_{\to}^{\alpha_-;\Psi}(\alpha.\homrlambdaXatYdotZ y \beta t; x))} y \beta)[\phi,\punitrefl \alpha, \psi, s/y, b/\beta]
      &= (\homrappXtoYatZ
        {(\textrm{ind}_{\to}^{\alpha_-;\Psi}(\alpha.\homrlambdaXatYdotZ y \beta t; x)[[\phi,\punitrefl \alpha, \psi]])} s b)\tag{Definition}\\
      &= (\homrappXtoYatZ {\homrlambdaXatYdotZ y \beta t[\phi,\psi]} s b) \tag{Induction}\\
      &= t[\phi,\psi,s/y,b/\beta]\tag{Hom$\beta$}
    \end{align*}
  \item $\Phi = \beta, y, \Phi$ case is similar to previous.
  \end{itemize}
  
  Next $\eta$.
  \begin{itemize}
  \item $\Phi = \alpha_- $ and $\Psi = \alpha_+$:
    \[ \textrm{ind}_{\to}^{\alpha_-;\alpha_+}(\alpha. s[\alpha/\alpha_-,\punitrefl \alpha/x,\alpha/\alpha_+]; x)
    = \punitelimkontZatABC {\alpha. s[\alpha/\alpha_-,\punitrefl \alpha/x,\alpha/\alpha_+]} {\alpha_1}{x}{\alpha_2}
    \]
    Which is equal to $s$ by the primitive unit $\eta$.
  \item $\Phi = \alpha_-$ and $\Psi = \Psi, y, \beta$:
    \begin{align*}
      s &= \homrappXtoYatZ {\homrlambdaXatYdotZ y \beta s} y \beta\tag{Hom $\eta$}\\
      &= \homrappXtoYatZ {(\textrm{ind}_{\to}^{\alpha_-;\Psi}(\alpha.\homrlambdaXatYdotZ y \beta {s[\punitrefl \alpha/x]}; x))} y \beta \tag{Induction}\\
      &= \textrm{ind}_{\to}^{\Phi;\Psi,y,\beta}(\alpha. s[\punitrefl \alpha/x]; x)\tag{Definition}
    \end{align*}
  \end{itemize}
\end{proof}

\section{Details of Formal Category Theory Examples}
\label{sec:appendix:examples}

Next, we provide some further details for some of the examples of the
formal category theory constructions and theorems from
Section~\ref{sec:examples}.

\begin{definition}[Synthetic Composition/Functoriality]
  We provide the definitions of the terms in Construction~\ref{construction:synthetic-composition}
  \begin{enumerate}
  \item Identity $\id = \pendlambdaXdotY \alpha \punitrefl \alpha$
  \item Composition $\textrm{comp} = \lambdaunary {\alpha_1} f {\alpha_2} \punitelimkontZatABC {\alpha. \homrallXYtoZ {\alpha_3} g g} {\alpha_1} f {\alpha_2}$
  \item Functoriality $\textrm{fctor}(F) = \lambdaunary {\alpha_1} f {\alpha_2} \punitelimkontZatABC {\alpha. \punitrefl{F\,\alpha}} {\alpha_1} f {\alpha_2}$
  \item Profunctoriality
    \[\textrm{prof}(R) = \lambdaunary {\alpha_1} f {\alpha_2} {\punitelimkontZatABC {\alpha. \homrlambdaXatYdotZ r {\beta_1} \homrlambdaXatYdotZ g {\beta_2} {\homlappXtoYatZ {(\punitelimkontZatABC {\beta. \homllambdaXatYdotZ r \alpha r}{\beta_1}{g}{\beta_2})} {r} {\alpha}}}{\alpha_1} f {\alpha_2}}\]

    We can also define left and right composition for profunctors by
    applying the profunctorial action to an identity morphism on one side or
    the other:

    \[ \textrm{lcomp}(R) = \lambdabinary {\alpha_1} f {\alpha_2} r {\beta} {\apptrinary {\textrm{prof}(R)} {\alpha_1} f {\alpha_2} r {\beta} {\punitrefl \beta} {\beta}}\]
    
    \[ \textrm{rcomp}(R) = \lambdabinary {\alpha} r {\beta_1} g {\beta_2} {\apptrinary {\textrm{prof}(R)}  {\alpha} {\punitrefl \alpha} {\alpha} r {\beta_1} g {\beta_2}}\]
  \end{enumerate}
  Associativity and unit follow by $\beta\eta$ for unit and homs.
\end{definition}

\begin{lemma}[Naturality]
  For any $t : \pendallXdotY {\alpha:\cat C}{\harrapp R \alpha \alpha}$,
  \[ \lambdaunary {\alpha_1} f {\alpha_2} {\appbinary {\textrm{lcomp}(R)}{\alpha_1} f {\alpha_2} {\pendappXtoY t {\alpha_2}} {\alpha_2}}
  =  \lambdaunary {\alpha_1} f {\alpha_2} {\appbinary {\textrm{rcomp}(R)}{\alpha_1}  {\pendappXtoY t {\alpha_1}} {\alpha_1} f {\alpha_2}}\]
\end{lemma}
\begin{proof}
  Expanding the definitions and applying $\beta$ reductions, both are equal to $\lambdaunary {\alpha_1} f {\alpha_2} \punitelimkontZatABC {\alpha. \pendappXtoY t \alpha} {\alpha_1} f {\alpha_2}$
\end{proof}

\begin{lemma}[Yoneda, Co-Yoneda]
  Let $\alpha^o : \cat C$ and $\pi : \posPresheaf {\cat C}$. Then
  \item (Yoneda) The profunctor $\homrallXYtoZ {\alpha'} {\punitinXfromYtoZ {\cat C} {\alpha'}{\alpha}} {\negPresheafAppPtoX \pi {\alpha'}}$ is isomorphic to $\negPresheafAppPtoX \pi \alpha$
  \item (Co-Yoneda) The profunctor $\tensorexistsXwithYandZ {\alpha'}{\punitinXfromYtoZ{}{\alpha}{\alpha'}} {\negPresheafAppPtoX \pi \alpha}$ is isomorphic to $\negPresheafAppPtoX \pi \alpha$
\end{lemma}
\begin{proof}
  We show Yoneda in detail.
  \begin{itemize}
  \item The left-to-right homomorphism is defined as
    \[ M = \lambdaunary {\alpha} {\phi} \pi {\homrappXtoYatZ {\phi} {\punitrefl {\alpha}} \alpha} \]
  \item The right-to-left homomorphism is defined as
    \[ N = \lambdabinary \pi x \alpha f \alpha' {\homlappXtoYatZ {\punitelimkontZatABC {\alpha. \homllambdaXatYdotZ x \pi x} {\alpha'} f {\alpha}} x \pi}\]
  \item First, right-to-left-to-right:
    \begin{align*}
      \lambdaunary \pi x \alpha {\appunary M {\alpha} {(\appunary N \pi x \alpha)} \pi}
      &= \lambdaunary \pi x \alpha {\appunary M {\alpha} {({\homrlambdaXatYdotZ f {\alpha'} {{\homlappXtoYatZ {\punitelimkontZatABC {\alpha. \homllambdaXatYdotZ x \pi x} {\alpha'} f {\alpha}} x \pi}}})} \pi}\\
      &= \lambdaunary \pi x \alpha {\homrappXtoYatZ {({\homrlambdaXatYdotZ f {\alpha'} {{\homlappXtoYatZ {\punitelimkontZatABC {\alpha. \homllambdaXatYdotZ x \pi x} {\alpha'} f {\alpha}} x \pi}}})} {\punitrefl \alpha} \alpha}\\
      &= \lambdaunary \pi x \alpha {{\homlappXtoYatZ {\punitelimkontZatABC {\alpha. \homllambdaXatYdotZ x \pi x} {\alpha'} {\punitrefl \alpha} {\alpha}} x \pi}}\tag{covhom$\beta$}\\
      &= \lambdaunary \pi x \alpha {{\homlappXtoYatZ {(\homllambdaXatYdotZ x \pi x)} x \pi}}\tag{unit$\beta$}\\
      &= \lambdaunary \pi x \alpha {x}\tag{contrahom$\beta$}\\
    \end{align*}
  \item Left-to-right-to-left
    \begin{align*}
      &\lambdaunary {\alpha} {\phi} \pi \appunary N {\pi} {({\appunary M \alpha \phi \pi})} {\alpha}\\
      &= \lambdaunary {\alpha} {\phi} \pi \homrlambdaXatYdotZ {f}{\alpha'} {\homlappXtoYatZ {\punitelimkontZatABC {\alpha. \homllambdaXatYdotZ x \pi x} {\alpha'} f {\alpha}} {({\appunary M \alpha \phi \pi})} \pi}\\
      &= \lambdaunary {\alpha} {\phi} \pi \homrlambdaXatYdotZ {f}{\alpha'} {\homlappXtoYatZ {\punitelimkontZatABC {\alpha. \homllambdaXatYdotZ x \pi x} {\alpha'} f {\alpha}} {(\appunary \phi \alpha {\punitrefl\alpha}{\alpha})} \pi}\\ 
      &= \lambdaunary {\alpha} {\phi} \pi \homrlambdaXatYdotZ {f}{\alpha'} \punitelimkontZatABC {\alpha. ({\homlappXtoYatZ {\punitelimkontZatABC {\alpha. \homllambdaXatYdotZ x \pi x} {\alpha} {\punitrefl \alpha} {\alpha}} {(\appunary \phi \alpha {\punitrefl\alpha}{\alpha})} \pi})} {\alpha'}{f}{\alpha}\tag{unit $\eta$}
      \\ 
      &= \lambdaunary {\alpha} {\phi} \pi \homrlambdaXatYdotZ {f}{\alpha'} \punitelimkontZatABC {\alpha. ({\homlappXtoYatZ {\homllambdaXatYdotZ x \pi x} {(\appunary \phi \alpha {\punitrefl\alpha}{\alpha})} \pi})} {\alpha'}{f}{\alpha}\tag{unit $\beta$} \\ 
      &= \lambdaunary {\alpha} {\phi} \pi \homrlambdaXatYdotZ {f}{\alpha'} \punitelimkontZatABC {\alpha. ({{(\appunary \phi \alpha {\punitrefl\alpha}{\alpha})}})} {\alpha'}{f}{\alpha}\tag{contrahom $\beta$} \\ 
      &= \lambdaunary {\alpha} {\phi} \pi \homrlambdaXatYdotZ {f}{\alpha'} {{(\appunary \phi {\alpha'} f {\alpha})}}\tag{unit $\eta$} \\ 
      &= \lambdaunary {\alpha} {\phi} \pi \phi \tag{covhom $\eta$} \\ 
    \end{align*}
  \end{itemize}
\end{proof}

\begin{lemma}[Fubini]
  We show two of the Fubini cases in detail:

  \begin{enumerate}
  \item $\homlallXYtoZ {\gamma} {(\tensorexistsXwithYandZ \beta {\harrapp P \gamma \beta} {\harrapp Q \beta \alpha})} {\harrapp S \gamma \delta} \cong \homlallXYtoZ \beta {\harrapp Q \beta \alpha} {\homlallXYtoZ \gamma {\harrapp P \gamma \beta} {\harrapp S \gamma \delta}}$
  \item $\pendallXdotY \alpha \homrallXYtoZ \beta {\harrapp P \alpha \beta} {\harrapp Q \alpha \beta} \cong \pendallXdotY \beta \homlallXYtoZ \alpha {\harrapp P \alpha \beta} {\harrapp Q \alpha \beta}$
  \end{enumerate}
\end{lemma}
\begin{proof}
  \begin{enumerate}
  \item This is a form of Currying isomorphism, as the $\lambda$ term
    makes clear:
    \begin{itemize}
    \item Left to Right
      \[ \lambdaunary \alpha h \delta \homllambdaXatYdotZ q \beta \homllambdaXatYdotZ p \gamma \homlappXtoYatZ h {(p,\beta,q)} \gamma  \]
    \item Right to Left
      \[ \lambdaunary {\alpha} k \delta \homllambdaXatYdotZ w \gamma \tensorelimWkontZ {w} {p,\beta,q. \homlappXtoYatZ {\homlappXtoYatZ k p \gamma} q \beta } \]
    \item Left to Right to Left
      \begin{align*}
        &\lambdaunary \alpha h \delta \homllambdaXatYdotZ w \gamma {\tensorelimWkontZ {w} {p,\beta,q. (\homlappXtoYatZ {\homlappXtoYatZ {(\homllambdaXatYdotZ q \beta \homllambdaXatYdotZ p \gamma {\homlappXtoYatZ h {(p,\beta,q)} \gamma})} q \beta }  p \gamma) }}\\
        &=\lambdaunary \alpha h \delta \homllambdaXatYdotZ w \gamma {\tensorelimWkontZ {w} {p,\beta,q. (\homlappXtoYatZ {{(\homllambdaXatYdotZ p \gamma {\homlappXtoYatZ h {(p,\beta,q)} \gamma})} }  p \gamma) }}\tag{contrahom$\beta$}\\
        &=\lambdaunary \alpha h \delta \homllambdaXatYdotZ w \gamma {\tensorelimWkontZ {w} {p,\beta,q. {{({\homlappXtoYatZ h {(p,\beta,q)} \gamma})} } }}\tag{contrahom$\beta$}\\
        &=\lambdaunary \alpha h \delta \homllambdaXatYdotZ w \gamma \homlappXtoYatZ h w \gamma\tag{tensor$\eta$}\\
        &=\lambdaunary \alpha h h\tag{contrahom$\eta$}
      \end{align*}
    \item Right to Left to Right
      \begin{align*}
        &\lambdaunary \alpha k \delta \homllambdaXatYdotZ q \beta \homllambdaXatYdotZ p \gamma \homlappXtoYatZ {(\homllambdaXatYdotZ w \gamma {\tensorelimWkontZ {w} {p,\beta,q. \homlappXtoYatZ {\homlappXtoYatZ k q \beta}  p \gamma }})} {(p,\beta,q)} \gamma\\
        &=\lambdaunary \alpha k \delta \homllambdaXatYdotZ q \beta \homllambdaXatYdotZ p \gamma { {\tensorelimWkontZ { {(p,\beta,q)}} {p,\beta,q. \homlappXtoYatZ {\homlappXtoYatZ k q \beta } p \gamma }}}\\
        &=\lambdaunary \alpha k \delta \homllambdaXatYdotZ q \beta \homllambdaXatYdotZ p \gamma {\homlappXtoYatZ {\homlappXtoYatZ k q \beta } p \gamma}\\
        &=\lambdaunary \alpha k \delta \homllambdaXatYdotZ q \beta { {\homlappXtoYatZ k q \beta}}\\
        &=\lambdaunary \alpha k \delta k\\
      \end{align*}
    \end{itemize}
  \item This isomorphism relates left and right homs. Unlike the
    previous cases, the isomorphism is of types, not sets/profunctors.
    \begin{itemize}
    \item Left to right
      \[\lambda X. \pendlambdaXdotY \beta \homllambdaXatYdotZ p \alpha {\appunary X \alpha p \beta}\]
    \item Right to left
      \[ \lambda Y. \pendlambdaXdotY \alpha \homrlambdaXatYdotZ p \beta {\homlappXtoYatZ {\pendappXtoY Y \beta} p \alpha} \]
    \item Left to right to left
      \begin{align*}
        &\lambda X. \pendlambdaXdotY \alpha \homrlambdaXatYdotZ p \beta {\homlappXtoYatZ {\pendappXtoY {(\pendlambdaXdotY \beta \homllambdaXatYdotZ p \alpha {\appunary X \alpha p \beta})} \beta} p \alpha}\\
        &= \lambda X. \pendlambdaXdotY \alpha \homrlambdaXatYdotZ p \beta {\homlappXtoYatZ { {(\homllambdaXatYdotZ p \alpha {\appunary X \alpha p \beta})}} p \alpha}\tag{nat.elt.$\beta$}\\
        &= \lambda X. \pendlambdaXdotY \alpha \homrlambdaXatYdotZ p \beta {\appunary X \alpha p \beta}\tag{contrahom$\beta$}\\
        &= \lambda X. \pendlambdaXdotY \alpha {\pendappXtoY X \alpha}\tag{contrahom$\eta$}\\
        &= \lambda X. X\tag{nat.elt.$\eta$}\\
      \end{align*}
    \item The other case is similar.
    \end{itemize}
  \end{enumerate}
\end{proof}

\begin{lemma}[Equivalent Definitions of Adjoints]
  We show that given a morphism of profunctors
  \[ \homunary \alpha {\punitinXfromYtoZ{} {L\alpha} \beta} \beta {\punitinXfromYtoZ {} \alpha {R \beta}}\]
  we can extract a unit natural transformation $\eta : \pendallXdotY
  \alpha \punitinXfromYtoZ {}{\alpha}{R(L\alpha)}$ and vice-versa.
\end{lemma}
\begin{proof}
  The construction is exactly the ordinary proof but formalized in \vett{} syntax.
  Given the morphism of profunctors $M$, we define the unit by evaluating at the identity:
  \[ \pendallXdotY \alpha \appunary M \alpha {\punitrefl \alpha} \alpha \]
  and given the unit $\eta$, we can define a morphism of profunctors
  by composing the unit with the functorial lift of the input:
  \[ \homunary \alpha {f} \beta \textrm{comp}(\pendappXtoY \eta \alpha, \textrm{fctor}(R)(f))\]
  That this is an isomorphism follows by a similar argument to the
  proof of the Yoneda lemma.
\end{proof}

\section{Details of Semantics}
\label{sec:appendix:semantics}

In this section, we provide the full descriptions of the universal
properties in a virtual equipment corresponding to each connective in
\vett{}.

\begin{definition}[Universal Properties for Category Connectives]
  Let $\mathcal V$ be a virtual equipment.
  \begin{enumerate}
  \item Let $C$ be a small object, then a contravariant presheaf object $\mathcal P^- C$ is an object with natural isomorphism
    $\mathcal V_o(A, \mathcal P^- C) \cong \{ R \in V_h \,|\, s(R) = C \wedge t(R) = A \}$
  \item Let $C$ be a small object, then a covariant presheaf object $\mathcal P^+ C$ is an object with natural isomorphism
    $\mathcal V_o(A, \mathcal P^+ C) \cong \{ R \in V_h \,|\, t(R) = C \wedge s(R) = A \}$
  \item Let $R$ be a horizontal arrow, then a tabulator $\int R$ is an object with natural isomorphism
    $\mathcal V_o(A, \int R) \cong \sum_{f : {\mathcal V}_o(A, s(R))}\sum_{g : {\mathcal V}_o(A, t(R))}{\mathcal V}_2(\cdot;f;g;R)$
  \item A nullary product is an object $1$ with natural isomorphism $\mathcal V_o(A,1) \cong 1$
  \item A binary product of $B$ and $C$ is an object $B \times C$ with natural isomorphism $\mathcal V_o(A,B \times C) \cong {\mathcal V}_o(A,B)  \times {\mathcal V}_o(A,C) $
  \end{enumerate}
\end{definition}

\renewcommand{\vec}{\overrightarrow}
\begin{definition}[Universal Properties for Set Connectives]
  Let $\mathcal V$ be a virtual equipment.
  \begin{enumerate}
  \item A unit $U_C$ for an object $C$ is a horizontal arrow $U_C$ with $s(U_C) = t(U_c) = C$ with natural isomorphism
    $\mathcal V_2(\vec P, U_c, \vec Q; f; g; R) \cong \mathcal V_2(\vec P, \vec Q; f; g; R)$
  \item A tensor of horizontal arrows $P$ and $Q$ where $t(P) = s(Q)$ is a horizontal arrow $P \odot Q$ with $s(P\odot Q) = s(P)$ and $t(P \odot Q) = t(Q)$ with natural isomorphism
    $\mathcal V_2(\vec R, P \odot Q,\vec S; f; g; T) \cong \mathcal V_2(\vec R, P, Q,\vec S; f; g; T)$.
  \item A covariant hom of $P$ and $Q$ where $t(P) = t(Q)$ is a horizontal arrow $P \triangleright Q$ with $s(P\triangleright Q) = s(Q)$ and $t(P \triangleright Q) = s(P)$ with natural isomorphism
    $\mathcal V_2(\vec R; f; \id; P \triangleright Q) \cong \mathcal V_2(\vec R, P; f; \id; Q)$
  \item A contravariant hom of $P$ and $Q$ where $s(P) = s(Q)$ is a horizontal arrow $P \triangleleft Q$ with $s(P\triangleleft Q) = t(Q)$ and $t(P \triangleleft Q) = t(P)$ with natural isomorphism
    $\mathcal V_2(\vec R; \id; g; P \triangleleft Q) \cong \mathcal V_2(Q, \vec R; \id; g; P)$
  \item A nullary product for an object $C$ is a horizontal arrow $1_C$ with $s(1_C) = t(1_C) = C$ with natural isomorphism
    $\mathcal V_2(\vec P; f; g; 1_C) \cong 1$
  \item A binary product of horizontal arrows $P$ and $Q$ where $s(P)= s(Q)$ and $t(P) = t(Q)$ is a horizontal arrow $P \times Q$ with $s(P \times Q) = s(P)$ and $t(P \times Q) = t(P)$ with natural isomorphism
    $\mathcal V_2(\vec R; f; g; P \times Q) \cong \mathcal V_2(\vec R; f; g; P) \times \mathcal V_2(\vec R; f; g; Q)$
  \end{enumerate}

  We require that in our models, units exist for all objects, tensors
  and homs overs small objects exist and all finite products exist. We
  additionally require that the choice of tensors, homs and products
  commute strictly with restrictions in that
  \begin{enumerate}
  \item $(P \odot Q)(f,g) = (P(f,\id) \odot Q(\id, g))$
  \item $(P \triangleright Q)(f,g) = (P(g, \id) \triangleright Q(f,\id))$
  \item $(P \triangleleft Q)(f,g) = (P(\id, g) \triangleleft Q(\id, f))$
  \item $1(f,g) = 1$
  \item $(P \times Q)(f,g) = (P(f, g) \times Q(f, g))$
  \end{enumerate}
  Note that these equations necessarily hold up to isomorphism, even
  if we do not require them to commute strictly.
\end{definition}


\subsection{Completeness}

Next we describe the syntactic properties of substitution that are
needed in order to prove the \emph{completeness} theorem, that is,
that the syntax of \vett{} presents a hyperdoctrine of virtual
equipments.

\begin{definition}[Syntactic Virtual Equipment]
  Fix a context $\Gamma$. Define a virtual equipment $\Syn^\Gamma$ as
  follows:
  \begin{enumerate}
  \item The vertical category $\Syn^\Gamma_o$ has categories $\Gamma
    \vdash \cat C \isaCat$ as objects, small categories as small
    objects and as arrows from $\cat C$ to $\cat D$ objects
    $\alpha:\cat C \vdash b : \cat D$ modulo renaming of the input
    variable. Composition is given by substitution and identity is the
    variable.
  \item The horizontal arrows are the sets $\alpha:\cat C; \beta:\cat D
    \vdash R$ (up to renaming $\alpha$ and $\beta$) with source $\cat C$
    and target $\cat D$.
  \item Note that composable strings $\vec R$ of horizontal arrows are
    in bijection with contexts $\Phi$. Then we can define a 2-cell
    $\Syn^\Gamma_2(\Phi, a, b, S)$ to be an element
    $\Gamma\pipe\Phi\vdash s : S[a/\alpha;b/\beta]$.

    Composition is defined by substitution $t[\phi]$ as substitutions
    are in bijection with the ``sequences of 2-cells'' used in the
    definition of a virtual equipment. Associativity says that
    $t[\phi][\psi] = t[\phi[\psi]]$ where the composition $\phi[\psi]$
    is defined below and corresponds exactly to the associativity rule
    in a virtual equipment. The unit is the variable, and they are
    unital as $x[s/x] = s$ and $s[\vec {x/x}] = s$.
    
  \item Restriction along vertical arrows is given by substitution
    $R(a,b) = R[a/\alpha;b/\beta]$. This is strictly associative and
    unital, and the cartesian cell from $R$ to $R[a/\alpha;b/\beta]$
    is just the identity $x:R[a/\alpha;b/\beta] \vdash x :
    R[a/\alpha;b/\beta]$.
  \end{enumerate}
\end{definition}

\begin{definition}
  \label{lem:vertical-sq-cat}
  We define the vertical composition of transformation substitutions
  $\phi[\psi]$ inductively on $\phi$.
  \begin{align*}
    (a/\alpha)[b/\beta] &= a[b/\beta]/\alpha\\
    (\phi_1,t/x,a/\alpha)[\psi_1\jnctx\psi_2] &= \phi_1[\psi_1], t[\psi_2],a[d^+\psi_2]
  \end{align*}
  This covers all cases by \cref{lem:inversion}.

  We define the vertical identity $\id_\Phi$ by induction on $\Phi$
  \begin{align*}
    \id_{\alpha:\cat C} &= \alpha/\alpha\\
    \id_{\Phi,x:R,\alpha:\cat C} &= \id_{\Phi},x/x,\alpha/\alpha
  \end{align*}

  By induction this is seen to be associative:
  \[ \phi[\psi][\sigma] = \phi[\psi[\sigma]] \]
  and unital
  \[ \id_{\Phi}[\phi] = \phi = \phi[\id_{\Psi}] \]
\end{definition}
\fi
\end{document}